\documentclass[12pt]{amsart}

\usepackage{amssymb}
\usepackage{amsmath}

\usepackage{dsfont}

\usepackage{array}
\usepackage{longtable}
\usepackage{url}
\usepackage{enumerate}
\usepackage{amsmath}
\usepackage{amssymb}
\usepackage{amsthm}
\usepackage{amsfonts}
\usepackage{tikz,pgf}
\usepackage{graphicx}
\usepackage{caption}
\usepackage{subcaption}

\newcommand{\sm}{\setminus}

\newcommand{\diag}[1]{\text{diag}(#1)}

\newcommand{\sse}{\subseteq}

\newcommand{\spn}{\text{span}}
\newcommand{\cspn}{\overline{\spn}}

\newcommand{\AF}{\mathcal{A}\mathcal{F}}
\newcommand{\nifty}{n\to\infty}

\newcommand{\mc}[1]{\mathcal{#1}}
\newcommand{\mf}[1]{\mathfrak{#1}}

\newcommand{\emp}{\emptyset}

\newcommand{\C}{\mathbb{C}}
\newcommand{\N}{\mathbb{N}}
\newcommand{\Q}{\mathbb{Q}}
\newcommand{\R}{\mathbb{R}}

\newcommand{\Z}{\mathbb{Z}}

\newcommand{\al}{\alpha}
\newcommand{\be}{\beta}
\newcommand{\ga}{\gamma}
\newcommand{\Ga}{\Gamma}
\newcommand{\de}{\delta}
\newcommand{\De}{\Delta}
\newcommand{\ep}{\varepsilon}
\newcommand{\eps}{\epsilon}
\newcommand{\ze}{\zeta}
\renewcommand{\th}{\theta}

\newcommand{\ka}{\kappa}
\newcommand{\la}{\lambda}
\newcommand{\La}{\Lambda}
\newcommand{\si}{\sigma}

\newcommand{\om}{\omega}

 \theoremstyle{plain} 
 \newtheorem{Theorem}{Theorem}[section]
 \newtheorem*{Theorem*}{Theorem}
 \newtheorem{Lemma}[Theorem]{Lemma}
 \newtheorem{Proposition}[Theorem]{Proposition}
 
 \newtheorem{Corollary}[Theorem]{Corollary}
 
 \theoremstyle{definition} 
 \newtheorem{Definition}[Theorem]{Definition}
 \newtheorem{Remark}[Theorem]{Remark}

\newtheorem{Notation}[Theorem]{Notation}
\newtheorem*{Notation*}{Notation}

\begin{document}

\title[Spectral Triples on Effros-Shen algebras]{Spectral Triples on a non-standard presentation of Effros-Shen AF algebras} 

  \author{Konrad Aguilar}
  \email{konrad.aguilar@pomona.edu}
\thanks{The first author is partially supported by NSF grant DMS-2316892.}
\address{Department of Mathematics and Statistics, Pomona College\\610 N College Ave\\Claremont, CA 91711\\USA}
             
 \author{Samantha Brooker}
  \email{bsamantha24@vt.edu}
\address{Department of Mathematics, Virginia Tech\\225 Stanger Street\\Blacksburg, VA\\24061\\USA}

 \author{Jack Spielberg}
 \email{spielberg@asu.edu}
\address{School of Mathematical and Statistical Sciences, Arizona State University\\901 S Palm Walk\\Tempe, AZ 85281\\USA}


\maketitle
\begin{abstract}
The Effros-Shen algebra corresponding to an irrational number $\theta$ can be described by an inductive sequence of direct sums of matrix algebras, where the continued fraction expansion of $\theta$ encodes the dimensions of the summands, and how the matrix algebras at the $n$th level fit into the summands at the $(n+1)$th level. In recent work, Mitscher and Spielberg present an Effros-Shen algebra as the $C^*$-algebra of a category of paths -- a generalization of a directed graph -- determined by the continued fraction expansion of $\theta$. With this approach, the algebra is realized as the inductive limit of a sequence of infinite-dimensional, rather than finite-dimensional, subalgebras. In the present work, we define a spectral triple in terms of the category of paths presentation of an Effros-Shen algebra, drawing on a construction by Christensen and Ivan. This article describes categories of paths, the example of Mitscher and Spielberg, and the spectral triple construction.
\end{abstract}

\section{Introduction}

In the realm of Noncommutative Geometry \`a la Connes \cite{Con94} and Noncommutatiave  Metric Geometry \`a la Rieffel \cite{Rie04}, there have been many recent advances with respect to inductive limits of C*-algebras \cite{FLLP24, FLP22, BCFL23, FRG19, CI06, AAAKM23}. Of particular importance has been the spectral triple structure, and AF algebras including the family of Effros-Shen algebras \cite{ES80} have been at the forefront of some of these investigations \cite{CI06, AL15, AAAKM23}. In other recent work \cite{MS22}, Mitscher and Spielberg present these Effros-Shen algebras as C*-algebras of categories of paths, a notion introduced by Spielberg
which provides a fairly general framework for studying C*-algebras built from combinatorial
data, such as the higher-rank graph algebras introduced by Kumjian and
Pask   \cite{KP00, Spi14}. With this approach, an Effros-Shen algebra is obtained as the inductive
limit of a sequence of infinite-dimensional subalgebras whose structure comes
from the (generalized) infinite path space of the underlying graphical object. 
With an interest in cross-pollination between the areas of
noncommutative  geometry  and    C*-algebras of categories of paths, we are
motivated to endow these AF algebras with a spectral triple structure that reflects
Mitscher and Spielberg?s alternative perspective.

Our strategy is motivated by the work of  Christensen and Ivan \cite{CI06} in which they provide a recipe for constructing a
spectral triple on a unital AF algebra with a faithful state. Their construction
utilizes the inductive sequence of finite-dimensional subalgebras,  but it
can be adapted to suit our needs, as we will see. Indeed, as the realization of Effros-Shen algebras from \cite{MS22} is given by an inductive sequence of infinite-dimensional C*-subalgebras rather than finite-di\-men\-sional, much of the work and novelty of this article involves finding and utilizing finite-di\-mensional approximations among these infinite-dimensional approximations arising from the categories of path construction in order to build our spectral triple.

We now outline the content of this article. Section 2 describes the objects of
interest: Effros-Shen C*-algebras, spectral triples, and Spielberg?s categories of paths
and their groupoids. This includes a summary of the spectral triple construction
given by Christensen and Ivan, which introduces notation to be used in later sections.
The section ends with a description of the category of paths identified by Mitscher
and Spielberg with a particular Effros-Shen algebra $\mathcal{AF}_\theta$.
 
We begin Section 3 by giving a general idea of our approach to constructing our
spectral triple. The rest of the section details this process: first we decompose the
GNS Hilbert space of  $\mathcal{AF}_\theta$ induced by its unique faithful tracial state into a direct
sum of subspaces according to the structure of the category of paths of \cite{MS22}. Then
we describe the action of $\mathcal{AF}_\theta$ on these finite-dimensional subspaces, and finally we define the spectral triple that is the main objective of this work.

\section{Preliminaries}
\subsection{Effros-Shen 
Algebras}

Let $\theta$ be an irrational real number. By Elliott's theorem, there is a unique simple unital AF algebra $\AF_{\theta}$ such that $K_0(\AF_{\theta}) = \Z + \theta \Z,$ $K_0(\AF_{\theta})_+ = (\Z + \theta \Z)\cap [0, \infty),$ and $[1]_0 = 1$. Effros and Shen give an explicit description in \cite{ES80}. The following is taken from \cite[VI.3]{Dav96}.

Suppose $\theta \in \R$ has simple continued fraction expansion $[a_0, a_1, \dots]$, meaning that $a_0 \in \Z$ and for $i \geq 1$, $a_i \in \N \setminus \{0\}$ such that
\begin{align*}
    \theta = \lim_{\nifty} [a_0, a_1, \dots, a_n] = \lim_{\nifty} a_0 + \cfrac{1}{a_1 + \cfrac{1}{a_2 + \cfrac{1}{a_3 + \cfrac{1}{\ddots + \cfrac{1}{a_n}}}}}.
\end{align*}

Let $p_n/q_n = [a_0, a_1, \dots, a_n]$, where $p_n, q_n$ are given by
$$\begin{array}{ll}
     p_0 = a_0 \quad& q_0 = 1  \\
     p_1 = a_0 a_1 + 1 \quad& q_1 = a_1\\ 
     p_n = a_n p_{n-1} + p_{n-2} \quad& q_n = a_n q_{n-1} + q_{n-2} \quad \text{ for } n \geq 2 \\
\end{array}$$
Thus,
\begin{align*}
    \begin{bmatrix}
    p_n & q_n \\ p_{n-1} & q_{n-1}
    \end{bmatrix}
    = \begin{bmatrix}
    a_n & 1 \\ 1 & 0
    \end{bmatrix}
    \begin{bmatrix}
    p_{n-1} & q_{n-1} \\ p_{n-2} & q_{n-2}
    \end{bmatrix}
\end{align*}
Then set $\mc{A}_n = M_{q_n} \oplus M_{q_{n-1}}$, and define the map $\alpha_{n-1, n}: \mc{A}_{n-1} \to \mc{A}_{n}$ by
\begin{align*}
    \alpha_{n-1,n}(a,b) = (\diag{a, \dots, a, b}, a)
\end{align*}
where there are $a_n$ copies of $a$ in the block diagonal matrix in the first coordinate of the result. Then $(\mc{A}_n, \alpha_{n,n+1})_{n=1}^{\infty}$ is an inductive sequence of finite-dimensional $C^*$-algebras with unital embeddings, and the inductive limit $$\AF_{\theta} := \lim_{\longrightarrow} \mc{A}_n$$ is the Effros-Shen algebra corresponding to $\theta$.

\subsection{Spectral Triples}

The following is taken from \cite{CI06}:
\begin{Definition}\label{spectral triple def}
A \textit{spectral triple} is a triple $(A, H, D)$ where $A$ is a unital $C^*$-algebra, $H$ is a Hilbert space which is a left $A$-module (ie, we regard $A$ as a subalgebra of $B(H)$ by way of a *-representation), and $D$ is an unbounded self-adjoint operator on $H$ such that
\begin{enumerate}[(a)]
    \item the set $B := \{a \in A: [D,a]$ is densely defined and extends to a bounded operator on $H\}$ is norm-dense in $A$, and
    
    \item $(1 + D^2)^{-1}$ is a compact operator.
\end{enumerate}
\end{Definition}

Let $A$ be a unital AF-algebra with faithful state $\tau$, and let $\{A_n\}_{n=0}^{\infty}$ be an increasing sequence of finite-dimensional subalgebras of $A$, with $A_0 = \C 1_A$, such that $A = \overline{\bigcup A_n}$. In \cite{CI06}, Christensen and Ivan give a way of constructing a spectral triple on $A$ using the inductive sequence. We review this construction here:

Let $H$ denote the GNS Hilbert space $L^2(A, \tau)$, that is, the Hilbert space completion of $A$ with inner product $\langle a, b \rangle = \tau(b^*a)$. Since $\tau$ is faithful, we can regard $A$ as a subalgebra of $B(H)$. Let $\xi$ be the unit vector in $H$ coming from the GNS construction, which is cyclic and separating for $A$, and satisfies $\tau(a) = \langle a \xi, \xi \rangle$ for all $a \in A$. Denote by $\eta : A \to H$ the map that sends $a \in A$ to $a \xi \in H$, which is injective since $\tau$ is faithful.

For each $n$, $H_n := \eta (A_n)$ is a finite-dimensional (hence closed) subspace of $H.$
Let $P_n$ be the orthogonal projection of $H$ onto $H_n$. On the $C^*$-algebra side of things, we get a continuous projection $\pi_n: A \to A_n$ by $$\pi_n(a)= \left( \eta|_{A_n} \right)^{-1}(P_n(\eta(a))).$$ This is well-defined since $\eta$ is injective, and bounded since $\eta(A_n)$ is finite dimensional. So we have that $\pi_n(a)\xi = P_n a \xi$ for all $a \in A$. 

Now, the sequence $(P_n)_{n=0}^{\infty}$ of increasing projections yields a sequence $(Q_n)_{n=1}^{\infty}$ of pairwise orthogonal projections by letting $Q_n = P_n - P_{n-1}$. For any sequence $(\alpha_n)_{n=0}^{\infty}$ increasing to infinity with $\alpha_0 = 0$, define an unbounded operator $D = \sum_{n=1}^{\infty} \alpha_n Q_n$ on $H$.

For any $n \in \N$, $a \in A_n$ commutes with $P_m$ for any $m \geq n$, and hence $a$ commutes with $Q_m$ for any $m > n$. So the commutator $[D,a] = \sum_{i=1}^n \alpha_i [Q_i, a]$ is densely defined, because it is defined on $\left( \bigcup A_i \right) \xi$, and it extends to a bounded operator on $H$.
This holds for all $a \in \bigcup A_i$, which is norm-dense in $A$, so definition \ref{spectral triple def}(a) is satisfied.

\subsection{Categories of Paths}

In \cite{MS22}, Mitscher and Spielberg present an alternative view of $\AF_{\theta}$ as the $C^*$-algebra of a groupoid built from a certain category of paths. From the  groupoid, they follow the standard construction of the groupoid $C^*$-algebra (the details of which can be found in, for instance, \cite{Put19}).

\begin{Definition}
    A category is \textit{small} if the collection of objects and the collection of morphisms are both sets.
\end{Definition}

\begin{Definition}[\cite{Spi14}]\label{def-cat-of-paths}
    A \textit{category of paths} $\Lambda$ is a small category satisfying for all $\mu, \nu, \ze \in \La$,
    \begin{enumerate}[(1)]
        \item $\mu\nu= \mu \ze$ implies $\nu = \ze$ \textit{(left-cancellation)}
        \item $\nu\mu= \ze\mu$ implies $\nu = \ze$ \textit{(right-cancellation)}
        \item $\mu\nu = s(\nu)$ implies $\mu = \nu = s(\nu)$ \textit{(no inverses)}
    \end{enumerate}
\end{Definition}

We can think of the category of paths $\La$ from \cite{MS22} as the path space of a countable directed graph (ie, it has countable vertex and edge sets), with some identifications. The set of \textit{edges} in the graph is denoted by $\Lambda^1$, and the set of \textit{vertices} by $\Lambda^0$. An element of $\Lambda$ is a path $\mu$, which we can write (not necessarily uniquely) as a finite word over the alphabet of edges: $\mu = \mu_1 \mu_2 \cdots \mu_n$, where each $\mu_i \in \Lambda^1$, and $\mu_i$ and $\mu_{i+1}$ are composable edges: the \textit{source} vertex $s(\mu_i)$  is equal to the \textit{range} vertex $r(\mu_{i+1})$.
\[
\begin{tikzpicture}
    \node (v0) {} 
    -- ++(2.5,0) node (v1) {$s(\mu_1) = r(\mu_2)$} 
    -- ++(2.5,0) node (v2) {} 
    -- ++(1.5,0) node (v3) {} 
    -- ++(2,0) node (v4) {};
    \draw[latex-] (v0) -- (v1) node[pos=0.5, above=1pt] (e1) {$\mu_1$};
    \draw[latex-] (v1) -- (v2) node[pos=0.5, above=1pt] (e2) {$\mu_2$};
    \draw[latex-] (v2) -- (v3) node[pos=0.5, above=1pt]  {$\dots$};
    \draw[latex-] (v3) -- (v4) node[pos=0.5, above=1pt] (en) {$\mu_n$};
\end{tikzpicture}
\]

\begin{Definition}[\cite{Spi14}]\label{def shift map}
    Let $\La$ be a category of paths. For any $\mu \in \La$, we denote the \textit{extensions} of $\mu$ by $\mu\La = \{\mu \nu : \nu \in s(\mu) \La\}$, where $v \La = \{\la \in \La : r(\la) = v\}$ for $v \in \La^0$. We call $\mu \nu$ a \textit{proper} (or \textit{nontrivial) extension} of $\mu$ if $\mu \neq \mu \nu$. If $\mu, \eta \in \La$ and $\mu \La \cap \eta \La \neq \emp$, we say that $\mu$ and $\eta$ have a \textit{common extension,} or that $\mu$ \textit{meets} $\eta$, which we write $\mu \Cap \eta$.
    
    If $\la = \mu \nu$, we say that $\nu$ is a \textit{suffix} (and $\mu$ is a \textit{prefix)} 
    of $\la$. 
    We say $\la$ and $\th \in \La$ have a \textit{nontrivial common suffix} if there exist $\la', \th' \in \La$ and $\nu \in \La\sm \La^0$ such that $\la = \la' \nu$ and $\th = \th' \nu$.

    We define the \textit{left shift} $\si^{\mu} : \mu\La \to s(\mu) \La$ by $\si^{\mu}(\mu \nu) = \nu$. If $E \sse \La$, we implicitly understand $\si^{\mu}(E)$ to mean $\si^{\mu}(E \cap \mu \La)$. 
\end{Definition}

The category of paths $\Lambda$ that we consider in this work is defined in \cite[Section 3]{MS22}, and is an amalgamation of two categories of paths, $\Lambda_1$ and $\Lambda_2$. Each can be represented by a graph on the set of vertices $\Lambda^0 = \{v_1, v_2, \dots\}$. First, we consider $\Lambda_1$:

\begin{figure}[h!]
\[
\begin{tikzpicture}[xscale=2,yscale=2]

\node (00) at (.5,0) [rectangle] {$\Lambda_1$:};

\node (10) at (1,0) [rectangle] {$v_1$};
\node (20) at (2,0) [rectangle] {$v_2$};
\node (30) at (3,0) [rectangle] {$v_3$};
\node (40) at (4,0) [rectangle] {$v_4$};
\node (50) at (5,0) [rectangle] {$v_5$};
\node (5half0) at (5.35,0) [rectangle] {$\boldsymbol{\cdots}$};

\draw[-latex,thick,dotted, red] (20) to[bend right=40] node[above] {$\alpha_1$} (10);
\draw[-latex,thick,dotted, red] (30) to[bend right=40] node[above] {$\alpha_2$} (20);
\draw[-latex,thick,dotted, red] (40) to[bend right=40] node[above] {$\alpha_3$} (30);
\draw[-latex,thick,dotted, red] (50) to[bend right=40] node[above] {$\alpha_4$} (40);

\draw[-latex,thick,dashed, blue] (20) to[bend left=40] node[below] {$\beta_1$} (10);
\draw[-latex,thick,dashed, blue] (30) to[bend left=40] node[below] {$\beta_2$} (20);
\draw[-latex,thick,dashed, blue] (40) to[bend left=40] node[below] {$\beta_3$} (30);
\draw[-latex,thick,dashed, blue] (50) to[bend left=40] node[below] {$\beta_4$} (40);

\end{tikzpicture}
\]
    \caption{$\La_1$}
    \label{fig:Lambda1}
\end{figure}
\noindent In $\Lambda_1$, the two colors of edges commute: 
for instance, $\alpha_1 \alpha_2 \beta_3 \alpha_4 = \beta_1 \alpha_2 \alpha_3 \alpha_4$. Thus, we can identify each path $\mu$ in $\Lambda_1$ by the number of $\alpha$-edges, the number of $\beta$-edges, and the vertex at the range of $\mu$,
so that if $\mu = \alpha_1 \alpha_2 \beta_3 \alpha_4$ as before, we can write $\mu = v_1 \alpha^3 \beta$. In fact, $\Lambda_1$ has a \textit{degree-functor} $\mf{d}: \Lambda_1 \to \N_0^2$, under which $\mf{d}(\alpha) = (1,0)$ and $\mf{d}(\beta) = (0,1)$, so that $\mf{d}(\alpha^p \beta^q) = (p,q)$ (ie, $\Lambda_1$ is a \textit{graph of rank 2}).

The second piece of our category of paths is $\Lambda_2$, which is determined by a sequence $(k_i)_{i=1}^{\infty}$ of nonnegative integers, with infinitely many $k_i \neq 0$. For each $i =1, 2, \dots,$ we have $k_i$ edges $\{ \gamma_i^{(1)}, \gamma_i^{(2)}, \dots, \gamma_i^{(k_i)} \}$ with $r(\gamma_i^{(j)}) = v_i$ and  $s(\gamma_i^{(j)}) = v_{i+1}$:
\begin{figure}[h]
\[
\begin{tikzpicture}[scale=2]
\node (00) at (.5,0) [rectangle] {$\Lambda_2$:};

\node (10) at (1,0) [rectangle] {$v_1$};
\node (20) at (2,0) [rectangle] {$v_2$};
\node (30) at (3,0) [rectangle] {$v_3$};
\node (40) at (4,0) [rectangle] {$v_4$};

\node (4half0) at (4.5,0) [rectangle] {$\boldsymbol{\cdots}$};
\node (1half1) at (1.5,.85) [rectangle] {$\boldsymbol{\vdots}$};
\node (2half1) at (2.5,.85) [rectangle] {$\boldsymbol{\vdots}$};
\node (23half1) at (3.5,.85) [rectangle] {$\boldsymbol{\vdots}$};

\draw[-latex,thick] (20) -- (10) node[pos=0.5,inner sep=0.5pt,above=1pt] {$\gamma^{(1)}_1$};
\draw[-latex,thick] (20) .. controls (1.75,.45) and (1.25,.45) .. (10) node[pos=0.5,inner sep=0.5pt,above=1pt] {$\gamma^{(2)}_1$};
\draw[-latex,thick] (20) .. controls (1.75,1.3) and (1.25,1.3) .. (10) node[pos=0.5,inner sep=0.5pt,above=1pt] {$\gamma^{(k_1)}_1$};

\draw[-latex,thick] (30) -- (20) node[pos=0.5,inner sep=0.5pt,above=1pt] {$\gamma^{(1)}_2$};
\draw[-latex,thick] (30) .. controls (2.75,.45) and (2.25,.45) .. (20) node[pos=0.5,inner sep=0.5pt,above=1pt] {$\gamma^{(2)}_2$};
\draw[-latex,thick] (30) .. controls (2.75,1.3) and (2.25,1.3) .. (20) node[pos=0.5,inner sep=0.5pt,above=1pt] {$\gamma^{(k_2)}_2$};

\draw[-latex,thick] (40) -- (30) node[pos=0.5,inner sep=0.5pt,above=1pt] {$\gamma^{(1)}_3$};
\draw[-latex,thick] (40) .. controls (3.75,.45) and (3.25,.45) .. (30) node[pos=0.5,inner sep=0.5pt,above=1pt] {$\gamma^{(2)}_3$};
\draw[-latex,thick] (40) .. controls (3.75,1.3) and (3.25,1.3) .. (30) node[pos=0.5,inner sep=0.5pt,above=1pt] {$\gamma^{(k_3)}_3$};

\end{tikzpicture}
\]
    \caption{$\Lambda_2$}
    \label{fig:Lambda2}
\end{figure}
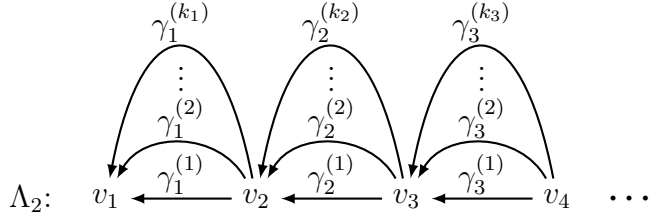

When amalgamated to form $\Lambda$, we arrive at the object depicted in Figure \ref{fig:Lambda}.
\begin{figure}[h]
\[
\begin{tikzpicture}[xscale=2.75,yscale=2]

\node (00) at (.5,0) [rectangle] {$\Lambda$:};

\node (10) at (1,0) [rectangle] {$v_1$};
\node (20) at (2,0) [rectangle] {$v_2$};
\node (30) at (3,0) [rectangle] {$v_3$};
\node (40) at (4,0) [rectangle] {$v_4$};

\node (4half0) at (4.35,0) [rectangle] {$\boldsymbol{\cdots}$};
\node (1half1) at (1.5,.85) [rectangle] {$\boldsymbol{\vdots}$};
\node (2half1) at (2.5,.85) [rectangle] {$\boldsymbol{\vdots}$};
\node (23half1) at (3.5,.85) [rectangle] {$\boldsymbol{\vdots}$};

\draw[-latex,thick] (20) -- (10) node[pos=0.5,inner sep=0.5pt,above=1pt] {$\gamma^{(1)}_1$};
\draw[-latex,thick] (20) .. controls (1.75,.45) and (1.25,.45) .. (10) node[pos=0.5,inner sep=0.5pt,above=1pt] {$\gamma^{(2)}_1$};
\draw[-latex,thick] (20) .. controls (1.75,1.3) and (1.25,1.3) .. (10) node[pos=0.5,inner sep=0.5pt,above=1pt] {$\gamma^{(k_1)}_1$};

\draw[-latex,thick] (30) -- (20) node[pos=0.5,inner sep=0.5pt,above=1pt] {$\gamma^{(1)}_2$};
\draw[-latex,thick] (30) .. controls (2.75,.45) and (2.25,.45) .. (20) node[pos=0.5,inner sep=0.5pt,above=1pt] {$\gamma^{(2)}_2$};
\draw[-latex,thick] (30) .. controls (2.75,1.3) and (2.25,1.3) .. (20) node[pos=0.5,inner sep=0.5pt,above=1pt] {$\gamma^{(k_2)}_2$};

\draw[-latex,thick] (40) -- (30) node[pos=0.5,inner sep=0.5pt,above=1pt] {$\gamma^{(1)}_3$};
\draw[-latex,thick] (40) .. controls (3.75,.45) and (3.25,.45) .. (30) node[pos=0.5,inner sep=0.5pt,above=1pt] {$\gamma^{(2)}_3$};
\draw[-latex,thick] (40) .. controls (3.75,1.3) and (3.25,1.3) .. (30) node[pos=0.5,inner sep=0.5pt,above=1pt] {$\gamma^{(k_3)}_3$};


\node (10) at (1,0) [rectangle] {$v_1$};
\node (20) at (2,0) [rectangle] {$v_2$};
\node (30) at (3,0) [rectangle] {$v_3$};
\node (40) at (4,0) [rectangle] {$v_4$};

\draw[-latex,thick,dotted,red] (20) .. controls (1.75,-.35) and (1.25,-.35) .. (10) node[pos=0.5,inner sep=0.5pt,above=1pt] {$\alpha_1$};
\draw[-latex,thick,dotted,red] (30) .. controls (2.75,-.35) and (2.25,-.35) .. (20) node[pos=0.5,inner sep=0.5pt,above=1pt] {$\alpha_2$};
\draw[-latex,thick,dotted,red] (40) .. controls (3.75,-.35) and (3.25,-.35) .. (30) node[pos=0.5,inner sep=0.5pt,above=1pt] {$\alpha_3$};

\draw[-latex,thick,dashed,blue] (20) .. controls (1.75,-.7) and (1.25,-.7) .. (10) node[pos=0.5,inner sep=0.5pt,below=1pt] {$\beta_1$};
\draw[-latex,thick,dashed,blue] (30) .. controls (2.75,-.7) and (2.25,-.7) .. (20) node[pos=0.5,inner sep=0.5pt,below=1pt] {$\beta_2$};
\draw[-latex,thick,dashed,blue] (40) .. controls (3.75,-.7) and (3.25,-.7) .. (30) node[pos=0.5,inner sep=0.5pt,below=1pt] {$\beta_3$};

\end{tikzpicture}
\]
    \caption{$\Lambda$}
    \label{fig:Lambda}
\end{figure}
While the $\alpha$'s and $\beta$'s are free to commute amongst themselves, the $\gamma$'s are like closed gates through which the $\Lambda_1$-edges cannot move.

Every path  $\lambda \in \Lambda \setminus \Lambda^0$ can be written in \textit{normal form} \cite[Lemma 3.3]{MS22} as the concatenation of paths alternately from $\Lambda_1$ and $\Lambda_2$: there exist unique $m \geq 1$ and $\theta_1, \dots, \theta_m \in \Lambda_1 \cup \Lambda_2$ such that 
\begin{itemize}
    \item $\lambda = \theta_1 \cdots \theta_m$
    \item $\theta_i \not\in \Lambda^0$
    \item $\theta_i \in \Lambda_j$ implies $\theta_{i+1} \in \Lambda_k$, where $\{j, k\} = \{1,2\}$.
\end{itemize}
If $\theta \in \Lambda_1 \cup \Lambda_2$, we let $|\theta|$ denote the number of edges in $\theta$, and if $\lambda = \theta_1 \cdots \theta_m$ is in normal form, let $|\lambda| = \sum_1^m |\theta_i|$.

We now describe the infinite paths of $\Lambda$. The infinite paths in $\Lambda_1$ are $\Lambda_1^{\infty} := \{ \alpha^p\beta^q : p+q = \infty\}$. The infinite paths of $\Lambda_2$ are infinite words $\Lambda_2^{\infty} := \{x_1 x_2 \cdots : x_i \in \Lambda_2^1,$ and $\forall n, \ x_1 x_2 \cdots x_n \in \Lambda_2\}.$ Of course, $\Lambda_2^{\infty}$ will be empty unless there is some $N \in \N$ such that $k_i \neq 0$ for all $i \geq N$. Now, paraphrasing  \cite[Theorem 3.9]{MS22}:
\begin{Theorem}
    Let $x$ be an infinite path in $\Lambda$. Then $x$ has one of the following forms:
    \begin{enumerate}[(1)]
        \item $x = \theta_1 \theta_2 \cdots$, where $\theta_1 \theta_2 \cdots \theta_n$ is in normal form for all $n$

        \item $x = \theta_1 \theta_2 \cdots \theta_{M-1} y$, where $\theta_1 \theta_2 \cdots \theta_{M-1}$ is in normal form, $\theta_{M-1} \in \Lambda_2$ if $M>1$, and $y \in \Lambda_1^{\infty}$

        \item $x = \theta_1 \theta_2 \cdots \theta_{M-1} y$, where $\theta_1 \theta_2 \cdots \theta_{M-1}$ is in normal form, $\theta_{M-1} \in \Lambda_1$ if $M>1$, and $y \in \Lambda_2^{\infty}$.
    \end{enumerate}
\end{Theorem}

We now describe the groupoid $\mc{G}$ arising from $\Lambda$. 

\begin{Definition}[\cite{Sims18}]
    A \textit{groupoid} is a set $\mc{G}$ together with a distinguished subset $\mc{G}^{(2)} \sse \mc{G} \times \mc{G}$, a multiplication map $(\al, \be) \mapsto \al \be$ from $\mc{G}^{(2)}$ to $\mc{G}$, and an inverse map $\ga \mapsto \ga^{-1}$ from $\mc{G}$ to $\mc{G}$ such that
    \begin{enumerate}[(1)]
        \item $(\ga^{-1})^{-1}= \ga$ for all $\ga \in \mc{G}$;
        \item if $(\al, \be)$ and $(\be, \ga)$ belong to $\mc{G}^{(2)}$, then $(\al \be, \ga)$ and $\al, \be \ga)$ belong to $\mc{G}^{(2)}$, and $(\al \be) \ga = \al (\be \ga)$; and

        \item $(\ga, \ga^{-1} \in \mc{G}^{(2)}$ for all $\ga \in \mc{G}$, and for all $(\ga, \eta) \in \mc{G}^{(2)},$ we have $\ga^{-1} (\ga \eta) = \eta$ and $(\ga \eta) \eta^{-1} = \ga$.
    \end{enumerate}
\end{Definition}

Let $X$ denote the collection of infinite paths in $\Lambda$. Then $$\mc{G} = \{ (\mu, \nu, x) : \mu, \nu \in v_1 \Lambda, x \in X, s(\mu) = s(\nu) = r(x) \} / \sim $$ where $\sim$ is the equivalence relation generated by $(\mu, \nu, \ep x) \sim (\mu \ep, \nu \ep, x)$. Thus a groupoid element is an equivalence class $[\mu, \nu, x]$, and the source and range maps for $\mc{G}$ are defined by $s([\mu, \nu, x]) = \nu x$ and $r([\mu, \nu, x]) = \mu x$, respectively. Multiplication is defined on ordered pairs in $\mc{G}^{(2)} := \{([\mu, \nu, \delta x], [\xi, \eta, \ep x]) \in \mc{G} \times \mc{G}: \nu \delta = \xi \ep\}$, so that \begin{align*}
    [\mu, \nu, \delta x] \cdot [\xi, \eta, \ep x] &= [\mu\delta, \nu \delta, x] \cdot [\xi \ep, \eta \ep, x]\\
    &= [\mu\delta, \eta \ep, x].
\end{align*}
The \textit{unit space} of the groupoid is $\mc{G}^{(0)} = \{[v_1, v_1, x] : x \in v_1 X\},$ which we identify with $X_1 := v_1 X$, the set of infinite paths ending at $v_1$. In general, we let $X_i = v_i X$.

For $\mu \in \Lambda$, let $Z(\mu) = \mu X$, the set of infinite paths that \textit{extend} $\mu$. Then 
\begin{equation}\label{compact-open-bis}
    \mc{B} := \left\{ Z(\mu) \setminus \cup_{i=1}^k Z(\nu_i) : \mu, \nu_1, \dots, \nu_k \in \Lambda\right\}
\end{equation} 
is a base consisting of compact open bisections for a totally disconnected, locally compact Hausdorff topology on $X$. For $m \ge 1$, let $Z_m = Z(v_m),$ and let $\mc{B}_{v_m} \equiv \mc{B}_{m} = \{E \in \mc{B} : E \sse Z(v_m)\}$.
For $F \subseteq Z_m$, we let $[\mu, \nu, F] = \{[\mu, \nu, x] : x \in F\}$. The collection $\{ [\mu, \nu, F] : \mu, \nu \in v_1 \La v_m, F \in \mc{B}_m \}$ is then a base of compact open bisections for a locally compact Hausdorff topology on $\mc{G}$ with respect to which $\mc{G}$ is an ample Hausdorff groupoid.

From there, the $C^*$-algebra $A := C^*(\mc{G})$ is generated by the characteristic functions of basic compact-open bisections, $\chi_{[\mu, \nu, F]}$, with adjoint defined by $\chi_{[\mu,\nu,F]}^* = \chi_{[\nu, \mu, F]},$ and multiplication defined by convolution:  the convolution of two such characteristic functions is \begin{align*}
    \chi_{[\mu, \nu, F]} * \chi_{[\delta, \ep, E]} = \chi_{[\mu, \nu, F] \cdot [\delta, \ep, E]},
\end{align*}
where $[\mu, \nu, F] \cdot [\delta, \ep, E] = \left\{[\mu, \nu, x] \cdot [\delta, \ep, y] : x \in F, y \in E\right\}$.


Here is how $\AF_{\theta}$ emerges from this setup: Let $\sigma \in \R \sm \Q$ have simple continued fraction expansion $\sigma = [c_0, c_1, \dots]$. Define a sequence of integers $(k_i)_{i=1}^{\infty}$ as follows (\cite{MS22} Theorem 5.12): Let $k_1$ be arbitrary, and for $p\geq 0$,
\begin{align*}
    k_i = \begin{cases}
        0 &: c_1 + c_3 + \cdots + c_{2p-1} + 2 < i < c_1 + c_3 + \cdots + c_{2p+1} + 2\\
        c_{2p} &: i = c_1 + c_3 + \cdots + c_{2p-1} + 2.
    \end{cases}
\end{align*}
Visually we may represent this by
\begin{align*}
    (k_i)_{i=1}^{\infty} = (k_1, c_0, \underbrace{0, \dots, 0}_{c_1 - 1}, c_2,  \underbrace{0, \dots, 0}_{c_3 - 1}, c_4, \underbrace{0, \dots, 0}_{c_5 - 1}, c_6, \dots).
\end{align*}
Let $\Lambda$ be the category of paths defined before, with $\Lambda_2$ determined by the sequence $(k_i)_{i=1}^{\infty}$ just described, and let $\mc{G}$ be the groupoid built from $\Lambda.$ Let $\theta = [0, 1, k_1, 1, k_2, 1, \dots]$. Then $C^*(\mc{G})$ is isomorphic to the Effros-Shen algebra $\AF_{\theta}$ \cite[ Theorem 7.2]{MS22}.

By \cite[Lemma 6.8 and Theorem 6.10]{MS22}, there exists a unique invariant Borel probability measure $\mf{m}$ on $\mc{G}^{(0)}$. On the basic cylinder sets $Z(\la_1 \cdots \la_h)$ and $Z(\la_1 \cdots \la_h) \sm Z(\la_1 \cdots \la_h \la_{h+1})$, where $\la_1, \dots, \la_h, \la_{h+1}\in \La^1$, $\mf{m}$ depends only on $h$. 

The state $\tau$ of \cite[Theorem 7.10]{MS22} on $C^*(\mc{G})$ coming from the invariant measure $\mf{m}$, is faithful and tracial \cite[Theorem 7.10]{MS22}. We use the data of $A = C^*(\mc{G})$ and $\tau$ to construct  the GNS Hilbert space $H$ with cyclic vector $\xi \in H$, and set $\eta: a \in A \mapsto a\xi \in H$, which is injective since $\tau$ is faithful. Then,
\begin{equation}
\begin{split}\label{inner-product-formula}
        \langle \eta(\chi_{[\de, \ep, E]}), \eta(\chi_{[\de', \ep', E']})\rangle &= \tau(\chi_{[\de', \ep', E']}^* * \chi_{[\de, \ep, E]})\\
        &= \tau(\chi_{[\ep', \de', E']} * \chi_{[\de, \ep, E]})\\
        &= \tau (\chi_{[\ep', \de', E'][\de, \ep, E]})\\
        &= \mf{m}([\ep', \de', E'][\de, \ep, E] \cap \mc{G}^{(0)}).
\end{split}
\end{equation}
Note that if $\de, \ep \in v_m \La$ and $E \sse Z_m$, then $[\de, \ep, E] \cap \mc{G}^{(0)} \neq \emp$ if and only if $\de = \ep$. So if $\de \neq \ep$ then $\tau(\chi_{[\de, \ep, E]}) = 0$.

Also, if $\de, \ep \in v_1 \La v_m$, $E, F \sse X_m$, and $E\cap F = \emp$, then $\eta(\chi_{[\de, \ep E]}) \perp \eta(\chi_{[\de, \ep F]})$, for we see
\begin{align*}
        \langle \eta(\chi_{[\de, \ep E]}), \eta(\chi_{[\de, \ep F]})\rangle &= \langle \chi_{[\de, \ep, E]} \xi, \chi_{[\de, \ep, F]} \xi\rangle\\
        &= \tau(\chi_{[\de, \ep, F]}^* *\chi_{[\de, \ep, E]})\\
        &= \tau(\chi_{[\ep, \de, F] \cdot [\de, \ep, E]})\\
        &= \tau(\chi_{[\ep, \ep, F\cap E]})\\
        &= \tau (\chi_{\emp})\\
        &=0.
    \end{align*}

\section{Constructing a Spectral Triple}
\subsection{Spectral Triples}
\label{sec spectral triple}
For $i \in \N \sm \{0\}$, let $\mc{G}_i$ be the subgroupoid of $\mc{G}$ generated by $\{ [\mu, \nu, x] \in \mc{G}: |\mu| = |\nu| \leq i\}$. Then \cite[Theorem 5.1]{MS22}  states:
\begin{Theorem}
    $C^*(\mc{G})$ is the limit of the inductive sequence $$C^*(\mc{G}_1) \to C^*(\mc{G}_2) \to \cdots$$ where the connecting maps are induced from the inclusion maps $C_c(\mc{G}_i) \hookrightarrow C_c(\mc{G}_{i+1})$.
\end{Theorem}
The subalgebras $C^*(\mc{G}_i)$ are not finite-dimensional, which is why we cannot directly apply Christensen and Ivan's arguments to get a spectral triple on $C^*(\mc{G})$ using this inductive sequence. However, $\overline{\eta(C^*(\mc{G}_i))} \ominus \overline{\eta(C^*(\mc{G}_{i-1}))}$ can be decomposed into the direct sum of finite-dimensional subspaces, on each of which we can define our Dirac operator to be constant. The following result shows how this works in general.

\begin{Proposition}\label{prop the point of D}
    Let $H$ be a Hilbert space with finite-dimensional subspaces $H_n$ such that $H = \bigoplus_{n=1}^{\infty} H_n$. Denote by $\mathds{1}$ the identity operator on $H$.
    Let $(c_n)_{n=1}^{\infty}$ be a sequence with $1 \leq c_n < c_{n+1}$, and $c_n \to \infty$, and suppose that for any $k \in \N$ there exists some $M_k>0$ such that if $|i-j| < k$ then $|c_i - c_j| \leq M_k$. Define an unbounded operator $D$ on $H$ by $D|_{H_n} = c_n \mathds{1}$. Let $T \in B(H)$ have the property that there is some $K \in \N$ such that for all $n \in \N$, \[TH_n \sse \bigoplus_{i = n-K}^{n+K} H_i. \]
    Then $[D,T]$ extends to a bounded operator on $H$.
\end{Proposition}
\begin{proof}
    Let $n \in \N$ be arbitrary, and let $\xi \in H_n$ be a unit vector. Then $T\xi = \sum_{i=n-K}^{n+K} a_i \xi_i,$ for some scalars $a_i$, and unit vectors $\xi_i$ with $\xi_i \in H_i$ for each $i$. Then $DT \xi = \sum_{i=n-K}^{n+K} a_i c_i \xi_i.$ And $TD\xi = T c_n \xi = \sum_{i=n-K}^{n+K} a_i c_n \xi_i$. Hence,
    \[
    DT\xi - TD \xi = \sum_{i=n-K}^{n+K} a_i (c_i - c_n) \xi_i
    \]
    so
    \[
    \|DT\xi - TD \xi\|^2 = \sum_{i=n-K}^{n+K} |a_i|^2 |c_i - c_n|^2 \|\xi_i\|^2
    \]
    Let $M = M_K$ as in the statement of the proposition. Then $|c_i - c_n| \leq M$ for $i$ between $n-K$ and $n+K$, thus
    \begin{align*}
        \|(DT - TD) \xi\|^2 &\leq \sum_{i=n-K}^{n+K} |a_i|^2 M^2 \|\xi_i\|^2\\
        &= M^2 \|T \xi\|^2\\
        &\leq M^2 \|T\|^2.
    \end{align*}
    Since $n$ was arbitrary, and $M$ does not depend on $n$ (it depends on $K$, which depends on $T$), this holds for $\xi \in \bigoplus_{i=1}^n H_i$, and the increasing union of these finite direct sums is dense in $H$. So, the commutator $[D,T]$ is bounded on a dense subspace of $H$, and thus extends to a bounded operator on $H$.
\end{proof}

In the following three subsections, we assemble the pieces required to apply Proposition \ref{prop the point of D} to $C^*(\mc{G})$ and the GNS Hilbert space $H$ in order to define an operator $D$ so that $(C^*(\mc{G}), H, D)$ is a spectral triple. 
First we go through the process of decomposing $H$ into finite-dimensional subspaces $H_{j,\ell,n}$ determined by the structure of the category of paths and its groupoid. We use families of partitions of the basic compact open sets for this purpose.
In Subsection 3.3, we define a subset $B_0$ of $A$ which generates $A$ as a $C^*$-algebra, and prove the main result:
\begin{Theorem*}[Theorem \ref{thm total bandwidth}]
    For all $a \in B_0$, there exists $j, m \ge 1$ such that if $j' > j$ and $\ell, n \ge m$, then \begin{equation*}
        a \cdot H_{j', \ell, n} \sse \bigoplus_{q=\ell-m}^{\ell+m} \bigoplus_{t = n-m}^{n+m} H_{j', q, t}.
    \end{equation*}
\end{Theorem*}
The proof of this takes place in several stages, which we point out along the way. 

In the final subsection, we define an unbounded operator $D$ and prove that we have successfully constructed a spectral triple.

\subsection{Subspaces} 
For $i \geq 0$, let $\Phi_i = \{\mu \in v_1 \Lambda : 1 \le |\mu| \leq i, \mu_{|\mu|} \in \Lambda_2^1\} \cup \{v_1\}.$ Recall the sequence $(k_i)_{i=1}^{\infty}$ where $k_i$ is the number of edges in $\Lambda_2$ from $v_{i+1}$ to $v_i$. Let $(k_{p_j})_{j=1}^{\infty}$ be the subsequence of $(k_i)_{i=1}^{\infty}$ consisting of all the nonzero terms. Let $k_0 = 0$, and $p_0 = 0$, so that $k_{p_j} = 0 \Leftrightarrow p_j = 0 \Leftrightarrow j=0$. Note that for $p_j \le i < p_{j+1}$,  $\Phi_i = \Phi_{p_j}$. 

Mitscher and Spielberg prove that $\mc{G}_i = \{[\mu\th, \mu'\th', x] \in \mc{G} : |\mu| = |\mu'| \le i$ and $\th, \th' \in \La_1\}$ \cite[Theorem 3.18]{MS22}. We further recharacterize $\mc{G}_i$ in the following:

\begin{Theorem}\label{thm new G_i desc}
    For $j > 1$ and for all $i$ with $p_j \le i < p_{j+1}$, $$\mc{G}_i = \{ [\mu\eta, \nu \th, x] \in \mc{G} : \mu, \nu \in \Phi_{p_j}, \eta, \th \in \La_1, \text{ and } \mf{d}(\eta) \perp \mf{d}(\th) \}.$$
\end{Theorem}

\begin{proof}
    Let $S$ denote the set on the righthand side of the equation. We begin by showing that  $S \sse \mc{G}_i$. Let $[\mu \eta, \nu \th, x] \in S$, and suppose without loss of generality that $|\mu| \le |\nu|$. Write $\eta = \eta' \eta''$ such that $|\eta'| = |\nu| - |\mu|$. Let $\mu' = \mu \eta'$. Then $|\mu'| = |\nu|,$ $\eta' \in \La_1$, and $[\mu' \eta'', \nu \th, x] = [\mu \eta, \nu\th, x]$.
    If $|\th| = 0$, then $\mf{d}(\th) = 0$ and so $\mf{d}(\eta'') \perp \mf{d}(\th)$. Otherwise, $|\eta| \ge |\th| > 0$, so since $\mf{d}(\eta) \perp \mf{d}(\th),$ it follows that $\eta = c^r, \th = d^s$, where $\{c, d\} = \{\al, \be\}$. Hence $\eta'' = c^{r'}$, so $\mf{d}(\eta'') \perp \mf{d}(\th)$.

    Now to show the reverse containment, we take $[\mu \eta, \nu \th, x] \in \mc{G}_i$, and, referring to \cite[Theorem 3.18]{MS22}, we can assume that $\eta, \th \in \La_1$.
    Let $m = \max\{j : \mu_j \in \La_2^1\}$, $n = \max\{j: \nu_j \in \La_2^1\}$. Write $\mu \eta = \mu' \eta'$ and $\nu \th = \nu' \th'$ so that $|\mu'| = m$ and $|\nu'| = n$. Then $\mu'_{|\mu'|}, \nu'_{|\nu'|} \in \La_2^1$, and since $|\mu| = |\nu| \le i < p_{j+1}$, it follows that $m, n \le p_j$. Thus $\mu', \nu' \in \Phi_{p_j}$.

    It also follows that $\eta', \th' \in \La_1$. If $\mf{d}(\eta') \perp \mf{d}(\th')$, then $[\mu' \eta', \nu' \th', x] \in S$ as desired. Otherwise, write $\eta' = \al^p\be^q$ and $\th' = \al^r \be^s$, let $a = \min\{p,r\}, \ b=\min\{q, s\}$. Then at least one of $(p-a), (r-a) = 0$, and at least one of $(q-b), (s-b) = 0$, so $\mf{d}(\al^{p-a} \be^{q-b}) \perp \mf{d}(\al^{r-a} \be^{s-b})$, and $[\mu' \eta', \nu' \th', x] = [\mu'\al^{p-a} \be^{q-b}, \nu'\al^{r-a} \be^{s-b}, \al^a \be^b x]$, which is in $S$.
\end{proof}

\begin{Definition} Define the following:
\begin{align*}
    \text{for  $j\geq 0$: }& S_j = \{(\mu, \nu) \in \Phi_{p_j} \times \Phi_{p_j} : \text{ at least one of } |\mu|, |\nu| = p_j \}\\
    \text{for  $j\geq 0:$ }& S_{j,0} = \{(\mu\eta, \nu\th) : (\mu, \nu) \in S_j, |\mu\eta| =|\nu\th| = p_j,\ \eta, \th \in \La_1,\\ & \text{ and } \mu\eta, \nu\theta \text{ have no common suffix} \} \\
    \text{for  $j\geq 1,\ell \geq 1$: }& S_{j,\ell} = \{(\mu\eta, \nu\th) : (\mu, \nu) \in S_j,\ |\mu\eta| = |\nu\th| = p_j + \ell,\\
    &\ \eta, \th \in \La_1,\ \mf{d}(\eta)\perp \mf{d}(\th)\}.
\end{align*}
\end{Definition}

\begin{Remark}\label{remark S_jl}
We point out three things about $S_{j,\ell}$: \begin{enumerate}[(1)]
    \item\label{remark S_jl 1} $S_{0, \ell} = \{(\al^{\ell}, \be^{\ell}), (\be^{\ell}, \al^{\ell})\}$.

    \item\label{remark S_jl 2} For $j > 0$ and $(\mu \eta, \nu \th) \in S_{j,0}$, at least one of $\eta, \th$ is trivial. Both are trivial if and only if $\mu, \nu \in \Phi_{p_j}$, and in this case, $\mu_{p_j} \neq \nu_{p_j}$.
    
    \item\label{remark S_jl 3} When $\ell \ge 1$ and $(\mu\eta, \nu\th) \in S_{j,\ell}$, $|\eta|, |\th| \ge 1$, so $\mf{d}(\eta) \perp \mf{d}(\th)$ implies that $\mu\eta$ and $\nu\th$ have no common suffix. Hence, for all $j, \ell \ge 0$, if $(\de, \ep) \in S_{j,\ell}$ then $\de, \ep$ have no common nontrivial suffix.
\end{enumerate}
\end{Remark}

\begin{Notation}
    Let $m(j,\ell) = p_j + \ell + 1$. Thus if $(\delta, \ep) $ with $|\de| = |\ep| = p_j + \ell,$  then $|\delta| = |\ep| = p_j + \ell = m(j,\ell)-1$, and so $s(\delta) = s(\ep) = v_{m(j,\ell)}$.
\end{Notation}

Now let $A = C^*(\mc{G})$, and  $A_j = C^*(\mc{G}_{p_j})$. Then $A_1 \sse A_2 \sse \cdots \sse A$, and $A$ is the inductive limit of $(A_j)_{j=1}^{\infty}$, by \cite[Theorem 5.1]{MS22}.
We let $H_1 = \overline{\eta(A_{1})},$ and for $j \ge 2$, $H_j = \overline{\eta(A_{j})} \ominus \overline{\eta(A_{j -1})}$.

\begin{Theorem}\label{thm H_j} ~
\begin{enumerate}
    \item\label{thm H_j 1} $H_1 = \cspn\{ \eta(\chi_{[\de, \ep, E]}) : (\de, \ep) \in S_{0,\ell} \cup S_{1, \ell}$ for some $\ell \ge 0, E \sse Z(s(\ep))$ compact open$\}$.

    \item\label{thm H_j 2} For $j \ge 2$,
    $H_j = \cspn\{\eta(\chi_{[\de, \ep, E]}) : (\de, \ep) \in S_{j,\ell}$ for some $\ell \ge 0, E \sse Z_{m(j,\ell)} $ compact open$\}.$
\end{enumerate}
\end{Theorem}
\begin{proof}
It follows immediately from Theorem \ref{thm new G_i desc} that $\mc{G}_{p_j} = \{ [\de, \ep, x] \in \mc{G} : (\de, \ep) \in S_{h, \ell}$ for some $h \le j, \ell \ge 0\}$. Thus (\ref{thm H_j 1}) holds.

Now assume $j \ge 2$.
Let $L_j = \spn\{\eta(\chi_{[\de, \ep, E]}) : (\de, \ep) \in S_{j,\ell}$ for some $ \ell \ge 0, E \sse Z_{m(j,\ell)} $ compact open$\}.$ 
    Then $L_j \sse \eta(C^*(\mc{G}_{p_j})) = \eta(A_{j})$, again by Theorem \ref{thm new G_i desc}.
    

    Next we show that $L_j \perp \eta(A_{j-1})$, which implies $\overline{L_j} \perp \overline{\eta(A_{j-1})}$, and so $\overline{L_j} \sse H_j$. 
    Let $(\de, \ep) = (\mu\eta, \nu\th) \in S_{j,\ell}$ for some $\ell \ge 0$ with $(\mu, \nu) \in S_{j}$. Let $E \sse Z_{m(j,\ell)}$ be compact open, so that $\eta(\chi_{[\de, \ep, E]}) \in L_j$. 

    The collection $\{\chi_{F} : F \sse \mc{G}_{p_{j-1}}$ is a compact open bisection$\}$ is a total set in $A_{j-1}$, and every compact open bisection in $\mc{G}_{p_{j-1}}$ is a finite disjoint union of sets of the form $[\de', \ep', E'] \sse \mc{G}_{p_{j-1}}$ where $E' \sse Z(s(\ep'))$ is compact open in $\mc{G}^{(0)}$. Thus, it suffices to show that $\eta(\chi_{[\de, \ep, E]}) \perp \eta(\chi_{[\de', \ep', E']})$ for some $(\de', \ep') = (\mu'\eta', \nu'\th') \in S_{j', \ell'}$  with $(\mu', \nu') \in S_{j'}$, where $j' \le j-1, \ell' \ge 0,$  and some compact open $E' \sse Z_{m(j',\ell')}$.

    By Equation \eqref{inner-product-formula},
    \begin{align*}
        \langle \eta(\chi_{[\de, \ep, E]}), \eta(\chi_{[\de', \ep', E']}) \rangle &= \mf{m}([\ep', \de', E'][\de, \ep, E] \cap \mc{G}^{(0)}).
    \end{align*}
    Let $Y = [\ep', \de', E'][\de, \ep, E] \sse \mc{G}$. If $Y \cap \mc{G}^{(0)} = \emp$ then $\mf{m}(Y \cap \emp) = 0$, so we shall assume for a contradiction that $Y \cap \mc{G}^{(0)} \neq \emp$. Thus $\de' \Cap \de$.

    If $\de' = \de$ then $Y =[\ep', \ep, E'\cap E]$. A set of the form $[\om_1, \om_2, F]$ can only intersect $\mc{G}^{(0)}$ if $\om_1 = \om_2$, so $\ep' = \ep$. But then $(\de, \ep) = (\de', \ep')$, so $\mu = \mu'$ and $\nu = \nu'$.. However, $|\mu'|, |\nu'| \le p_{j'} \le p_{j-1} < p_j$, and one of $|\mu|$ or $|\nu| = p_j$, so this is a contradiction.
    
    So, $\de' \Cap \de$ but $\de' \neq \de$. 
    Write $\de' = \phi_1 \cdots \phi_p$ and $\de = \psi_1 \cdots \psi_q$ in normal form. 
    Then by \cite[Theorem 3.4]{MS22}, $\phi_i = \psi_i$ for $i < \min\{p, q\}$.
    
    Note that $\mu'$ extends $\phi_1 \cdots \phi_{p-1}$, for either $\phi_p \in \La_1$ and so $\phi_p = \eta'$ and $\mu' = \phi_1 \cdots \phi_{p-1}$, or $\phi_p \in \La_2$, and so $\mu' = \de' = \phi_1 \cdots \phi_p$. By the same token, $\mu$ extends $\psi_1 \cdots \psi_{q-1}$.

    First, I claim that $|\mu'| < |\mu|$. For suppose $|\mu'| \ge |\mu|$. Then $|\mu| \leq p_{j'} < p_j$, which implies  that $|\eta| \ge 1$, so $\mu = \psi_1 \cdots \psi_{q-1}$ and $\eta = \psi_q$. Now, either $p > q$, or $p = q$.

    Suppose $p > q$. Then $\phi_q$ extends $\psi_q$ in $\La_1$, so $\mu'$ extends $\de$, and since $\psi_q \in \La_1$, $\mu'$ properly extends $\phi_1 \cdots \phi_q$, so $\mu'$ properly extends $\de$. Therefore write $\mu' = \de \ze$, $|\ze| \ge 1$. Then $\de' = \de \ze \eta'$, and 
    \begin{align*}
        Y &= [\ep', \de\ze\eta', E'][\de\ze\eta', \ep\ze\eta', \si^{\ze\eta'}(E)]\\
        &= [\ep', \ep\ze\eta', E\cap \si^{\ze\eta'}(E)].
    \end{align*}
    Then $\ep' = \nu' \th' = \ep \ze \eta'$. But then $\ze \eta'$ is a common suffix of $\nu'\th'$ and $\mu' \eta' = \de\ze\eta'$, contradicting the definition of $S_{j',\ell'}$.

    So, $p = q$, and $\phi_p = \eta'$, $\psi_p = \eta$ have a common extension in $\La_1$, say $\eta \ze = \eta' \ze'$ in $\La_1$. Since $\phi_p, \psi_p \in \La_1$, it follows that $\mu' = \phi_1 \cdots \phi_{p-1} = \psi_1 \cdots \psi_{p-1} = \mu$, so $\de = \mu \eta$ and $\de' = \mu \eta'$. Then
    \begin{align*}
        Y &= [\ep'\ze', \de'\ze', \si^{\ze'}(E')] [\de\ze, \ep\ze, \si^{\ze}(E)]\\
        &=  [\ep'\ze', \ep\ze, \si^{\ze'}(E')],
    \end{align*}
    so $\ep'\ze' = \ep \ze$, equivalently, $\nu' \th' \ze' = \nu \th \ze$. Since $|\mu| < p_j,$ $|\nu| = p_j$, and $\nu_{|\nu|} \in \La_2^1$. And since $\th' \ze', \th \ze \in \La_1$, $\nu' \th' \ze' = \nu \th \ze$ implies that $\nu' = \nu$. But $|\nu'| \le p_{j'} < p_j = |\nu|$, a contradiction.

    Therefore, $|\mu'| < |\mu|$. It then follows that $p \le q$. Moreover, we claim that $\mu$ extends $\de'$ nontrivially:

    \begin{enumerate}
        \item[\textbf{Case 1:}] Suppose $p=q$. Then $\phi_p, \psi_p$ have a common extension in $\La_j$ where $\phi_p, \psi_p \in \La_j$. 
    
    If $\phi_p \in \La_1$, then $\phi_p, \psi_p$ have a common externsion in $\La_1$. In particular, this means $\phi_p = \eta'$ and $\psi_p = \eta$. But then $\mu' = \phi_1 \cdots \phi_{p-1} = \psi_q \cdots \psi_{p-1} = \mu$, contradicting that $|\mu'| < |\mu|$.

    So, $\phi_p, \psi_p \in \La_2$. Then $\de = \mu$ and $\de' = \mu'$, and one of $\mu, \mu'$ extends the other. Since $|\mu'| < |\mu|,$ it follows that $\mu$ extends $\mu' = \de'$ nontrivially.

    \item[\textbf{Case 2:}] Suppose $p < q$. Then $\psi_p$ extends $\phi_p$, so $\mu$ extends $\de'$. Moreover, this is a proper extension: if $|\eta'| = 0$ then $\de' = \mu',$ and $\mu \neq \mu'$. If $|\eta'| > 0$, then $\de'$ ends with an edge from $\La_1$, while $\mu$ ends with an edge from $\La_2$, so $\mu \neq \de'$.
    \end{enumerate}
    This finishes the proof of the claim, so we may write $\mu = \de' \ze$, where $|\ze| \ge 1$. Then $\de = \de'\ze \eta$, and
    \begin{align*}
        Y &= [\ep' \ze \eta, \de' \zeta \eta, \si^{\ze \eta}(E')] [\de' \ze \eta, \nu \th, E]\\
        &= [\ep' \ze \eta, \nu \th, \si^{\ze \eta}(E') \cap E].
    \end{align*}
    So $\ep' \ze \eta = \nu \th$, ie, $\nu' \th' \ze \eta = \nu \th$. By assumption $\eta, \th \in \La_1$, and $\ze_{|\ze|} \in \La_2^1$. Hence $|\nu| \ge 1$, and hence $\nu_{|\nu|} \in \La_2$. In fact, $\ze_{|\ze|}$ is the last edge from $\La_2$ in $\nu' \th' \ze \eta$, and $\nu_{|\nu|}$ is the last edge from $\La_2$ in $\nu \th$, so $\nu = \nu' \th' \ze$, and $\th  = \eta$. But $\nu\th, \mu\eta$ cannot have a nontrivial common suffix, so this implies $|\th| = |\eta| = 0$, and $(\de, \ep) = (\mu, \nu)$. But $\mu = \de' \ze$ and $\nu = \nu'\th'\ze$, ie, $\mu,\nu$ have a nontrivial common suffix, contradicting the definition of $S_{j,\ell}$.

    Therefore, we conclude that $Y \cap \mc{G}^{(0)} = \emp$, and so $L_j \perp \eta(A_{j-1})$, as desired. This completes the proof that $\overline{L_j} \sse H_j$.

    To show that $H_j \sse \overline{L_j}$, consider $\xi = \eta(\chi_{[\mu\eta, \nu\th, E]})$, a typical element of a total set for $\eta(A_j)$. If one of $|\mu|, |\nu| = p_j$, then $\xi \perp \eta(A_{j-1})$ so $\xi \in H_j$. If both $|\mu|, |\nu| < p_j$, then $\xi \in \eta(A_{j-1})$ and so $\xi \perp H_j$.

    Now let $\xi \in H_j$ and $\eps > 0$ be arbitrary.
    Since $H_j \sse \eta(A_j)$, and $C := \spn\{\eta(\chi_F) : F = [\de, \ep, E] \sse \mc{G}_{p_j},$ compact open$\}$ is dense in $\eta(A_j)$, choose $\xi' \in C$ with $\|\xi - \xi'\| < \eps$. Write $\xi' = \sum_{i=1}^n b_i \eta(\chi_{F_i})$, where for each $i$, $F_i = [\mu_i \eta_i, \nu_i \th_i, E_i]$ with $(\mu_i \eta_i, \nu_i \th_i) \in S_{h_i, \ell_i}$ for some $h_i \le j, \ell_i \ge 0,$ and $E_i \sse Z_{m(h_i, \ell_i)}$ is compact open.
    We can then let $\mc{F}_1 = \{F_i :$ at least one of $|\mu_i|, |\nu_i| = p_j\}$ and $\mc{F}_2 = \{F_i :$ both $|\mu_i|, |\nu_i| < p_j\}$, and let $b_F = b_i$ if $F = F_i$. Then we may write $\xi' = \xi_1 + \xi_2$, where
    \begin{align*}
        \xi_1 &= \sum_{F \in \mc{F}_1} b_F \eta(\chi_F) \in L_j,\\
        \xi_2 &= \sum_{F \in \mc{F}_2} b_F \eta(\chi_F) \in \eta(A_{j-1}).
    \end{align*}
    Then $\xi_1 \in L_j \sse H_j$, and $\xi_2 \in \eta(A_{j-1}) \perp H_j$. Hence,
    \begin{align*}
        \eps^2 &> \|\xi - \xi'\|^2\\
        &= \|(\xi - \xi_1) - \xi_2\|^2\\
        &= \|\xi - \xi_1\|^2 + \|\xi_2\|^2,
    \end{align*}
    and so $\|\xi - \xi_1\| < \eps$. Thus $\xi \in \overline{L_j}$.
\end{proof}

Since $k_{p_j} < \infty$ for all $j \ge 1$, $\Phi_{p_j}$ is a finite set. For fixed $\ell \ge 0$, $S_{j,\ell}$ is also finite. For fixed $j, \ell$, there is a family of partitions $\{Q^{m(j,\ell)}_n\}_{n=-1}^{\infty}$ of $Z_{m(j,\ell)}$ each consisting of finitely many compact open sets, which we now describe. First, for $m \geq 1, n \geq 0$, define
\begin{align*}
    \Phi^m_n &= \{\mu \in v_m \La : 1 \le |\mu| \leq n, \mu_{|\mu|} \in \La_2^1\} \cup \{v_m\}.
\end{align*}
Note that $\Phi_0^m = \{v_m\}$, and $\Phi_n^m \sse \Phi_{n+1}^m$ for all $n \ge 0$. 

The following result will be useful later.

\begin{Proposition}\label{prop extension}
    If $\mu \in \Phi_{n_1}^m$ and $\nu \in \Phi_{n_2}^m$, and $\mu \Cap \nu$, then one of $\mu, \nu$ extends the other.
\end{Proposition}
\begin{proof}
    If either $\mu$ or $\nu = v_m$, then this is trivially true. So assume $|\mu|, |\nu| \geq 1$. Write both in normal form:
    \begin{align*}
        \mu &= \th_1 \th_2 \cdots \th_M\\
        \nu &= \eta_1 \eta_2 \cdots \eta_N.
    \end{align*}
    Note that $\th_M, \eta_N$ must be in $\La_2$ since $\mu_{|\mu|}$ and $\nu_{|\nu|} \in \La_2^1.$

    If $M \neq N$ then by \cite[Lemma 3.4]{MS22}, one of $\mu$ or $\nu$ extends the other. Otherwise, $M = N$, $\th_i = \eta_i$ for $i < M$, and $\th_M, \eta_M$ have a common extension in $\La_2$. Say without loss of generality that $|\th_M| \leq |\eta_M|$. Then since $\La_2$ is a 1-graph, $\eta_M$ extends $\th_M$, and hence $\nu$ extends $\mu$.
\end{proof}

\begin{Notation*}
For $m \ge 1$ and $\mu \in \La$, let $Z_m(\mu) = Z(v_m \mu)$.    
\end{Notation*}

\begin{Definition}\label{def Pmn}\cite[Lemma 5.3]{MS22}  Let $m \ge 1$ and $n \ge 0$. Define
\begin{align*}
    P^{(1)}_{m,n} &= \{Z_m(\al^{n+1}\be^j)\sm Z_m(\al^{n+1}\be^{j+1}) : j \leq n\}\\
    P^{(2)}_{m,n} &= \{Z_m(\al^{i}\be^{n+1})\sm Z_m(\al^{i+1}\be^{n+1}) : i \leq n\}\\
    P^{(3)}_{m,n} &= \{Z_m(\al^{i}\be^j \la) : i, j \leq n \leq i+j, \la \in v_{m+i+j}\La_2^1\}\\
    P^{(4)}_{m,n} &= \{Z_m(\al^{n+1}\be^{n+1})\}\\
    P_{m,n} &= \bigcup_{r=1}^4 P_{m,n}^{(r)}.
\end{align*}
\end{Definition}

Now we expand on the definition of $Q_n$ from \cite[Proposition 5.6]{MS22}. 

\begin{Definition}\label{def Qmn}
For $m \ge 1$, let $Q^m_{-1} = \{Z_m\}$, and for $n \ge 0$, define $Q_n^m$ by
\begin{align*}
    Q^m_n &= \bigcup_{\mu\in \Phi^m_n} \mu P_{m+|\mu|, n-|\mu|}.
\end{align*}
\end{Definition}
Then $Q^m_n$ is a finite partition consisting of compact-open sets which refines every set of the form $Z(\nu)$ and $Z(\nu) \sm Z(\nu \lambda)$ for $\nu \in v_m\Lambda$, $|\nu| \leq n$, and $\lambda \in s(\nu)\Lambda^1$ \cite[Proposition 5.6]{MS22}.

We now prove a result that allows us to shift our perspective when discussing these partitions.

\begin{Lemma}\label{lem equiv of Qs}
    Let $m \geq 1, n \geq 0$, and $t \ge 0$. If $\th \in \Phi_n^m$, then $\th Q_{n+t-|\th|}^{m+|\th|} = Z(\th) \cap Q_{n+t}^m$. 
\end{Lemma}

\begin{proof}
    Fix $\th \in \Phi_n^m$. If $|\th|=0$ then the result is trivial, so assume $|\th|\geq 1$.
    
    First suppose $A \in Q_{n+t-|\th|}^{m+|\th|},$ then $A = \phi A'$ for some $\phi \in \Phi_{n+t-|\th|}^{m+|\th|}$ and $A' \in P_{m+|\th|+|\phi|, n+t-|\th| - |\phi|}$. Then $|\th \phi| = |\th| + |\phi| \leq |\th| + n - |\th| + t = n+t$. The final edge $(\th \phi)_{|\th\phi|} = \phi_{|\phi|} \in \La_2^1$, and $r(\th\phi) = r(\th) = v_m$. So $\th\phi \in \Phi^m_{n+t}$. Hence $\th\phi A' = \th A \in Q_{n+t}^m$. And clearly $\th A \sse Z(\th) $.

    Now, suppose $A \in Z(\th) \cap Q_{n+t}^m$. Then $A = \phi A'$ for some $\phi \in \Phi_{n+t}^m$ and some $A' \in P_{m+|\phi|, n+t-|\phi|}$, and $\phi A' \sse Z(\th)$. It follows that $\phi \Cap \th$, and Proposition \ref{prop extension} implies that one of $\phi, \th$ must extend the other. 
    
    If $\phi$ extends $\th$, so that $\phi = \th \phi'$, then $|\phi'| = |\phi| - |\th| \leq n+t - |\th|$, and $\phi'_{|\phi'|} = \phi_{|\phi|} \in \La_2^1$, and $r(\phi') = s(\th) = v_{m+|\th|}$. So, $\phi' \in \Phi_{n+t-|\th|}^{m+|\th|}$. And $A' \in P_{m+|\phi|, n+t-|\phi|} = P_{m+|\th|+|\phi'|, n+t - |\th| - |\phi'|}$ since $|\th| + |\phi'| = |\phi|$. Therefore, $\phi A' = \th \phi' A' \in \th Q_{n+t-|\th|}^{m+|\th|}$.

    So, we shall assume $\phi$ does not extend $\th$, and deduce a contradiction. In particular, $\th \neq \phi$, so $\th$ extends $\phi$ nontrivially: $\th = \phi \th'$, $|\th'| \geq 1$. Write $\th'$ in normal form: $\th' = \psi_1 \cdots \psi_N$. Then $\psi_N \in \La_2$, since $\th \in \Phi_n^m$.

    If $A' \in P^{(1)}_{m+|\phi|, n+t-|\phi|} \cup P^{(2)}_{m+|\phi|, n+t-|\phi|}$, then $A' = Z(\ze) \sm Z(\ze \ka)$ for $\ze \in \La_1, \ka \in \La_1^1$, and $|\ze| \geq n+t-|\phi|+1$. Since $A' \sse Z(\th')$, $\ze$ and $\th'$ have a common extension. Since $\ze \in \La_1$ and $\psi_N \in \La_2$, it must be the case that $N > 1$, and $\psi_1$ extends $\ze$, by \cite[Lemma 3.4]{MS22}. But this is a contradiction because $|\psi_1| \leq |\th'| = |\th| - |\phi| \leq n-|\phi| < |\ze|$.

    If $A' \in P^{(3)}_{m+|\phi|, n+t-|\phi|}$ then $A' = Z(\ze)$ where $\ze = \al^i \be^j \la$, $i + j \geq n+t-|\phi|$ and $\la \in \La_2^1$. Once again, $\ze$ and $\th'$ have a common extension. If $N = 1$, then \cite[Lemma 3.4]{MS22} implies $\al^i\be^j$ extends $\psi_1$ in $\La_1$, since $\al^i \be^j \in \La_1$. In particular, $\psi_1 \in \La_1$, but $\psi_1 = \psi_N \in \La_2$, so we have a contradiction.

    If $N =2$, then $\psi_1 = \al^i \be^j$, and $\la$ and $\psi_2$ have a common extension in $\La_2$. If $|\psi_2| = 1$ then $\psi_2 = \la$, and so $\th' = \ze$. But $|\th'| = |\th| - |\phi| \leq n - |\phi|$, and $|\ze| = i +j + 1 \geq n+t -|\phi| + 1 > n-|\phi|$.

    If $N=2$ and $\psi_2$ properly extends $\la$, or if $N > 2$, then $\th'$ properly extends $\ze$, and so $|\th'| > |\ze|$, which is again a contradiction.

    If $A' \in P^{(4)}_{m+|\phi|, n+t-|\phi|}$ then $A' = Z(\ze)$ for $\ze = \al^{n+t-|\phi|+1} \be^{n+t-|\phi|+1}$, and $\ze$ and $\th'$ have a common extension. Since $\ze \in \La_1$ and $\psi_N \in \La_2$, $N > 1$. So, $\psi_1$ extends $\ze$. But $|\ze| \geq 2(n+t -|\phi|+1) > n-|\phi| \geq |\th'| > |\psi_1|$, so we have a contradiction.
\end{proof}

Recall that a partition $\mc{E}$ is said to \textit{refine} a set $A$ if $A$ can be written as a union of sets in $\mc{E}$. A partition $\mc{E}$ refines another partition $\mc{E}'$ if for all $E \in \mc{E}'$, $\mc{E}$ refines $E$. And finally, a set $A$ is said to refine $\mc{E}$ if there exists $E \in \mc{E}$ such that $A \sse E$.

\begin{Lemma}\label{lem Q_n+1 refines Q_n}
    Let $m \ge 1$ and $n \ge 0$. Then $Q_{n+1}^m$ refines $Q_n^m$.
\end{Lemma}
\begin{proof}
    Let $E = Z_m(\al^{n+1} \be^j) \sm Z_m(\al^{n+1}\be^{j+1}) \in P^{(1)}_{m,n}$ where $j \leq n$. Then
    \begin{align*}
        E &= Z_m(\al^{n+2} \be^j)\sm Z_m(\al^{n+2}\be^{j+1}) \cup \bigcup_{\la \in v_{m+n+2+j}\La_2^1} Z_m(\al^{n+1}\be^j \la).
    \end{align*}
    The set $Z_m(\al^{n+2} \be^j)\sm Z_m(\al^{n+2}\be^{j+1}) \in P_{m,n+2}^{(1)}$, and for each $\la \in v_{m+n+2+j}\La_2^1$, the set $Z_m(\al^{n+1}\be^j \la) \in P_{m,n+1}^{(3)}$.
    
    If $E \in P^{(2)}_{m,n}$, a similar decomposition to the one for $E \in P_{m,n}^{(1)}$ suffices.

    If $E = Z_m(\al^{i}\be^{j} \la) \in P^{(3)}_{m,n}$ for $i, j \le n \le i+j$ and some $\la \in \La_2^1$ with $r(\la) = s(v_m \al^i \be^j)$, then $i, j \le n+1$, so if $n+1 \le i+j$ then $E \in P^{(3)}_{m,n+1}$. If $i+j = n$ then $v_m \al^i \be^j \la \in \Phi_{n+1}^m$. Let $m' = m + i + j + 1$, and let $\th = v_m \al^i \be^j \la$. Then $E = \th Z_{m'}$, and $P_{m',0} = \{Z_{m'}(\al)\sm Z_{m'}(\al \be), Z_{m'}(\be)\sm Z_{m'}(\al \be), Z(\al\be)\} \cup \{Z_{m'}(\ka) : \ka \in v_{m'}\La_2^1\}$. Hence $Z_{m'} = \cup_{F \in P_{m',0}} F$, and $E = \cup_{F \in P_{m',0}} \th F$. Each $\th F \in Q_{n+1}^{m}$, so $E$ is refined by $Q_{n+1}^m$.
    
    If $E = Z_m(\al^{n+1} \be^{n+1}) \in P^{(4)}_{m,n}$, then 
    \begin{multline*}
        E = Z_m(\al^{n+2} \be^{n+1}) \sm Z_m(\al^{n+2} \be^{n+2}) \cup Z_m(\al^{n+1} \be^{n+2}) \sm Z_m(\al^{n+2} \be^{n+2}) \\ \cup \bigcup_{\la \in v_{m+2n+2}\La_2^1} Z_m(\al^{n+1}\be^{n+1}\la) \cup Z_m(\al^{n+2}\be^{n+2}).
    \end{multline*}
    Finally, if $\th \in \Phi_n^m$ and $E \in P_{m+|\th|, n-|\th|}$, then $\th \in \Phi_{n+1}^m$ as well, and $\th E$ is refined by sets in $\th Q^{m+|\th|}_{n+1-|\th|} \sse Q_{n+1}^m$.
\end{proof}

\begin{Notation} We set the following:
\begin{enumerate}[(1)]
    \item If $\mc{E}$ is a partition of a set $Z$, and $A \sse Z$ is refined by $\mc{E}$, let $A \cap \mc{E}$ denote the partition $\{ E \in \mc{E} : E \sse A\}$.

    \item In the special case where $E \in Q_{n}^{m}$, let $E^m_{n+1} \equiv E_{n+1} := E \cap Q_{n+1}^m$.
\end{enumerate}
\end{Notation}

\begin{Notation}
    If $\De = (\de, \ep) \in S_{j,\ell}$ and $F \sse Z_{m(j,\ell)}$ is compact open, let $\eta_{\De, F} = \eta(\chi_{[\de, \ep, F]})$.
\end{Notation}

\begin{Definition}\label{def K spaces}
   Fix $j \ge 1, \ell \ge 0$, and $n \ge -1$. Let $m = m(j,\ell)$. For $\De = (\de, \ep) \in S_{j, \ell}$ and $E \in Q_n^m$, define
       $$K_{\De}(E, E\cap Q_{n+1}^m) \equiv K_{\De}(E)_{n+1} := \spn\{\eta_{\De, F} : F \in E_{n+1}^m \} \cap \{\eta_{\De, E}\}^{\perp}.$$
\end{Definition}
Since $E_{n+1}^m$ is a finite partition of $E$ for any $E \in Q_n^m$, $K_{\De}(E)_{n+1}$ is a finite dimensional subspace of $H$. These spaces are constructed so that 
\begin{itemize}
    \item if $E, F \in Q_n^m$, $E \neq F$, then $K_{\De}(E)_{n+1} \perp K_{\De}(F)_{n+1}$, and

    \item if $E \in Q_n^m$ and $E' \in Q_t^m$ for some $t > n$, then $K_{\De}(E)_{n+1} \perp K_{\De}(E')_{t+1}$.
\end{itemize}

\begin{Definition}\label{def H_j l n}
    For $j \ge 2$, $\ell \ge 0$, and $n \ge -1$, let $m = m(j,\ell)$, and let
    \begin{align*}
        H_{j, \ell, n} &= \bigoplus_{\De \in S_{j,\ell}} \bigoplus_{E \in Q_n^{m}} K_{\De}(E)_{n+1},\\
        H_{j,\ell,-2} &= \bigoplus_{\De \in S_{j,\ell}} \C \eta_{\De, Z_m}.
    \end{align*}
    For $j =1$, we need to modify these formulas to account for $S_{1,\ell}$ and $S_{0,\ell}$ (see Theorem \ref{thm H_j}(\ref{thm H_j 1}) for comparison): for $n \ge -1$, let
    \begin{align*}
        H_{1,\ell,n} &= \bigoplus_{\De \in S_{0,\ell}} \bigoplus_{E \in Q_n^{m(0,\ell)}} K_{\De}(E)_{n+1} \oplus \bigoplus_{\De \in S_{1,\ell}} \bigoplus_{E \in Q_n^{m(1,\ell)}} K_{\De}(E)_{n+1},\\
        H_{1, \ell, -2} &= \bigoplus_{\De \in S_{0,\ell}} \C \eta_{\De, Z_{m(0,\ell)}} \oplus \bigoplus_{\De \in S_{1,\ell}} \C \eta_{\De, Z_{m(1,\ell)}}.
    \end{align*}
\end{Definition}

Since $S_{j,\ell}$ and $Q_{n}^{m(j,\ell)}$ are finite for all $j \ge 0, \ell \ge 0, n \ge -1$, each $H_{j,\ell,n}$ (and each $H_{j,\ell,-2}$) is finite-dimensional. Thus we have decomposed $H_j$ into a direct sum of finite-dimensional pieces:

\begin{Theorem}\label{thm H_j decomposition} Let $j \ge 1$. 
Then 
    \begin{align*}
        H_j = \bigoplus_{\ell \ge 0} \bigoplus_{n \ge -2} H_{j,\ell,n}.
    \end{align*}
\end{Theorem}

\begin{proof}
    We must show that for all $\ell \ge 0$, all $\De \in S_{j,\ell}$ (and all $\De \in S_{0,\ell},$ if $j = 1$), and all compact open sets $E \sse Z_{m(j,\ell)}$, that $\eta_{\De, E} \in \bigoplus_{n \ge -2} H_{j,\ell,n}$. If $E$ is compact open, then it is a finite union of basic compact open sets, ie, elements of $\mc{B}$ defined in equation \eqref{compact-open-bis}. So, $\eta_{\De,E}$ is a finite sum of vectors of the form $\eta_{\De,B}$ for $B \in \mc{B}$. If $B = Z_m(\mu) \sm \cup_{i=1}^j Z_m(\nu_i) \in \mc{B}$, take $n = \max\{ |\mu|, |\nu_i| : i = 1, \dots, j\}$. Then $Q^m_n$ refines $Z_m(\mu)$ and every $Z_m(\nu_i)$, so $\eta_{\De, B}$ can be written as a linear combination over $\{\eta_{\De, E} : E \in Q^m_n\}$.

    So, fix $j\ge 1, \ell \ge 0,$ and $\De \in S_{j,\ell}$, and let $m = m(j,\ell)$. We prove for all $n \ge -2$ that
    \begin{equation*}
        \spn\{\eta_{\De,F}: F \in Q^m_{n+1} \} \sse \C \eta_{\De, Z_m} \oplus \bigoplus_{i=-1}^n \bigoplus_{E \in Q_i^m} K_{\De}(E)_{i+1}.
    \end{equation*}
    We proceed by induction. Let $V_n$ denote the space on the righthand side.
    First, when $n = -2$, the only set in $Q^m_{-1}$ is $Z_m$, and $\spn\{\eta_{\De, Z_m}\} = \C \eta_{\De, \Z_m} \sse V_{-2}$.

    Now assume that $n \ge -2$ and that $\spn\{\eta_{\De, F} : F \in Q^m_{n+1} \} \sse V_n$. Then
    \begin{align*}
        V_{n+1} &= \C \eta_{\De, Z_m} \oplus \bigoplus_{i=-1}^{n+1} \bigoplus_{E \in Q_i^m} K_{\De}(E)_{i+1}\\
        &= \C \eta_{\De, Z_m} \oplus \bigoplus_{i=-1}^n \bigoplus_{E \in Q_i^m} K_{\De}(E)_{i+1} \oplus \bigoplus_{E \in Q_{n+1}^m} K_{\De}(E)_{n+2}\\
        &\supseteq \spn\{\eta_{\De,E} : E \in Q_{n+1}^m\} \oplus \bigoplus_{E \in Q_{n+1}^m} K_{\De}(E)_{n+2}.
    \end{align*}
    For each $E \in Q_{n+1}^m$, 
    \begin{align*}
        V_{n+1} &\supseteq \spn \{ \eta_{\De, E} \} \oplus \left( \spn \{ \eta_{\De,F} : F \in E_{n+2}^m \} \cap \{ \eta_{\De,E} \}^{\perp} \right)\\
        &= \spn\{\eta_{\De,F} : F \in E_{n+2}^m\},
    \end{align*}
    hence $$V_{n+1} \supseteq \spn\{ \eta_{\De,F} : F \in \bigcup_{E \in Q_{n+1}^m} E_{n+2}^m \}.$$
    Now, $Q_{n+2}^m$ refines by $Q_{n+1}^m$, meaning that if $F \in Q_{n+2}^m$ then there exists $E \in Q_{n+1}^m$ such that $F \sse E$. Thus $Q_{n+2}^m = \bigcup_{E \in Q_{n+1}^m} E_{n+2}^m$. Therefore, $V_{n+1} \supseteq \spn\{\eta_{\De,F} : F \in Q_{n+2}^m\}$, as desired.
\end{proof}

To summarize, $H_j$ decomposes into finite dimensional subspaces indexed by $\ell$ and $n$; 
$j$ tells us there is a barrier that imposes some inflexibility on the first $p_j$ edges of $\delta$ and $\ep$ in an element $\eta(\chi_{[\delta, \ep, E]}) \in H_j$.
Paired with $j$, the parameter $\ell$ tells us how long $\delta$ and $\ep$ are, since $|\delta| = |\ep| = p_j + \ell$. And $n$ indicates the maximum ``depth'' of the cylinder sets that can appear in the third coordinate. 

On the model of Christensen and Ivan, we define a net $(c(j,\ell,n))$ of scalars tending to infinity, and set $D|_{H_{j,\ell,n}}$ to be multiplication by $c(j,\ell,n)$. We use Proposition \ref{prop the point of D} to show for all $i$ and all $a \in A_i$ that $[D,a]$ is densely defined and bounded on $H$, and thus extends to a bounded operator.

\subsection{Bandwidth Considerations}

In this subsection, we prove the main result that we require for our spectral triple construction. The ultimate goal is to prove that elements $a$ of a dense subset of $A$ only shift the finite-dimensional direct summands $H_{j,\ell,n}$ of $H$ by a bounded amount.

The first theorem identifies and justifies the form of the elements $a$ that we consider.

\begin{Theorem}\label{thm total set for A}
    Let $B_0 = \{\chi_{[\mu c^r, \nu d^s, Z(s(\nu d^s))]} : (\mu, \nu) \in S_j$ for some $j \ge 0$, $|\mu c^r| = |\nu d^s|, \ \{c, d\} = \{\al,\be\} \}$. Then finite convolutions of elements in $\spn (B_0)$ comprise a total set in $A = C^{*}(\mc{G})$.
\end{Theorem}

\begin{proof}
    The set $\{\chi_{[\de, \ep, Z(\th)]} : |\de| = |\ep|,$ and $\th \in s(\ep) \La \}$ is a total set for $A$. In $\mc{G}$, $[\de, \ep, Z(\th)] = [\de\th, \ep\th, Z(s(\th))]$,
    and we can assume $\de \th = \mu \eta$ and $\ep \th = \nu \ze$ where $\mu, \nu \in \Phi_{p_j}$ for some $j \ge 1, \eta, \ze \in \La_1$. Therefore, the set $\{\chi_{[\mu\eta, \nu\ze, Z(s(\ze)]} :  |\mu\eta|= |\nu\ze|$ with $(\mu, \nu) \in S_{j}$ for some $j \ge 1, \eta, \ze \in \La_1\}$ is a total set in $A$.
    
    Let $\chi_{[\mu\eta, \nu\ze, Z(s(\ze)]}$ be an element of that set.
    Let $m = |\mu\eta| + 1.$ If $\eta = \al^a \be^b$ and $\ze = \al^c \be^d$, then
    \begin{align*}
        [\mu \al^a \be^b, \nu\al^c \be^b, Z_m] =  [\mu \be^b, \mu \al^b, Z_{m-a}] [\mu \al^{a+b}, \nu \be^{c+d}, Z_m] [\nu \be^c, \nu\al^c, Z_{m-d}].
    \end{align*}
    Thus, $\chi_{[\mu\eta, \nu\ze, Z(s(\ze)]} = \chi_{[\mu \be^b, \mu \al^b, Z_{m-a}]} * \chi_{[\mu \al^{a+b}, \nu \be^{c+d}, Z_m]} * \chi_{[\nu \be^c, \nu\al^c, Z_{m-d}]}$. This completes the proof.
\end{proof}

\begin{Notation}
    For $\De = (\de, \ep)$ with $|\de| = |\ep|$, let $a_{\De} = \chi_{[\de, \ep, Z(s(\ep))]}$.
\end{Notation}

\begin{Theorem}\label{thm bandwidth}
    Let $j \ge 1$, $\ell \ge 0$, and let $(\mu, \nu) \in S_j$. Let $\De = (\de, \ep) = (\mu \al^r, \nu \be^s)$, where $|\de| = |\ep| = p_j + \ell$ (ie, $r = p_j - |\mu| + \ell,$ $s = p_j - |\nu| + \ell$). Set $m = m(j,\ell)$.
        Let $j' > j, \ell' \ge m,$ take $(\mu', \nu') \in S_{j'}$, and $\De' = (\de', \ep') = (\mu' c^{r'}, \nu' d^{s'}) \in S_{j', \ell'}$, so $\{c, d\} = \{\al, \be\}$, and set $m' = m(j', \ell')$. Let $n \ge m$, and take $E \in Q_n^{m'}$.

    Then
    \begin{align*}
        a_{\De} \cdot K_{\De'}(E)_{n+1} \sse \bigoplus_{q=\ell' - m}^{\ell'+m} \ \bigoplus_{\Gamma \in S_{j', q}} \ \bigoplus_{t=n-m}^{n+m} \ \bigoplus_{G \in Q_t^{m(j',q)}} K_{\Gamma}(G)_{t+1}.
    \end{align*}
\end{Theorem}

\begin{Remark}
    Theorem \ref{thm bandwidth} remains true if we replace $\De$ with $(\mu \be^r, \nu \al^s)$, but we omit the proof of this.
\end{Remark}

In short, Theorem \ref{thm bandwidth} says that we know how far one of our finite-dimensional building blocks $K_{\De}(E)_{n}$ is spread when acted upon by an element of $B_0$. This forms the basis of the main result of this subsection:

The bulk of this section is spent proving Theorem \ref{thm bandwidth}.

The next two results tell us the relationship between $\ep$ and $\de'$, and $\nu$ and $\mu'$.

\begin{Proposition}\label{prop ep meets de'}
    If $\De, \De'$ are as in the statement of Theorem \ref{thm bandwidth}, and $E \in Q_n^{m'}$, and $a_{\De} \cdot \eta_{\De', E'} \neq 0$, then $\ep \Cap \de'$.
\end{Proposition}
\begin{proof}
    Suppose $\chi_{[\de, \ep, Z_m]} \cdot \eta_{\De', E} = \eta(\chi_{[\de, \ep, Z_m] [\de', \ep', E]}) \neq 0$. Then $[\de, \ep, Z_m]  [\de', \ep', E] \neq \emp$, so $\ep Z_m \cap \de' E \neq \emp$. In particular, $\ep \La \cap \de' \La \neq \emp,$ so $\ep \Cap \de'$.
\end{proof}

\begin{Proposition}\label{prop extension 2}
Let $j \geq 1, \ell \geq 0, m = m(j,\ell)$, and $j' \geq j, \ell' \geq m$. Suppose $(\mu, \nu) \in S_j, (\mu', \nu') \in S_{j'}$, and $(\de, \ep) = (\mu \al^r, \nu \be^s)$ with $|\de| = |\ep| = p_j + \ell, (\de', \ep') = (\mu'c^{r'}, \nu' d^{s'}) \in S_{j', \ell'}$.
    If $\ep \Cap \de'$, then either $\mu' = \nu$, or $\mu'$ is a nontrivial extension of $\ep$.
\end{Proposition}

\begin{proof}
    First, if both $|\mu'|$ and $|\nu| = 0$ then $\mu' = \nu = v_1$, so the result holds. If $|\nu| = 0$ then $\mu'$ is a nontrivial extension of $\nu = v_1$.
    If $|\mu'| = 0$ and $|\nu| \geq 1$, then since $\nu \be^s \Cap c^{r'}$ and $\nu_{|\nu|} \in \La_2^1$ while $c^{r'} \in \La_1$, $\nu$ must be a proper extension of $c^{r'}$. But then $r' = p_{j'} + \ell'$, while $|\nu| \leq p_j \leq p_{j'}$, so we have a contradiction.

    Now assume $|\nu|, |\mu'| \geq 1$.
    Let $\ze \in \ep\La \cap \de' \La = \nu\al^r \La \cap \mu' d^{s'} \La \sse \nu \La \cap \mu' \La$. It follows that $\nu \Cap \mu'$. Since $\nu \in \Phi_{p_{j}}$ and $\mu' \in \Phi_{p_{j'}}$, one of $\nu$ and $\mu'$ must extend the other, by \ref{prop extension}.
    Suppose then, for a contradiction, that $\nu$ is a proper extension of $\mu'$, so that $\nu = \mu' \xi$ for some $\xi$ with $|\xi| \ge 1$. Then $\xi_{|\xi|} = \nu_{|\nu|} \in \La_2^1$.
    Then $\ze = \ep \ze' = \nu\be^s \ze' = \mu' \xi \be^s\ze'$, and at the same time $\ze = \de' \ze'' = \mu' c^{r'} \ze''$. So, $\si^{\mu'}(\ze) = \xi \be^s \ze' = c^{r'} \ze''$, and it follows that $\xi \Cap c^{r'}$. Now, $c^{r'} \in \La_1$, and $\xi_{|\xi|} \in \La_2$, so it must be the case that $\xi$ extends $c^{r'}$. But $|\xi| < |\nu| = p_j + \ell$. And $|c^{r'}| = r' = p_{j'} - |\mu'| + \ell' \ge \ell' > p_j + \ell + 1$. So, $\xi$ cannot extend $c^{r'}$.

    Thus we have that $\mu'$ extends $\nu$, and write $\mu' = \nu \xi$. If $|\xi| = 0$ then we have that $\mu' = \nu$. Otherwise, $|\xi| \ge 1,$ and $\xi_{|\xi|} = \mu'_{|\mu'|} \in \La_2^1$. Since $\ep \Cap \de',$ and $\ep = \nu \be^s$ and $\de' = \mu' c^{r'} = \nu \xi c^{r'}$, we see that $\xi \Cap \be^s$. And since $\xi_{|\xi|} \in \La_2$ while $\be^s \in \La_1$, it must be that $\xi$ extends $\be^s$: $\xi = \be^s \xi'$. But now we have $\mu' = \nu \xi = \nu \be^s \xi' = \ep \xi'$. That is, $\mu'$ extends $\ep$. And $|\xi'| \ge 1$ since it contains at least one edge -- its last edge -- from $\La_2$. Therefore $\mu'$ extends $\ep$ nontrivially.
\end{proof}

In light of the preceding proposition, we treat the case where $\mu'$ properly extends $\ep$ separately from the case where $\mu' = \nu$.

\begin{Lemma}\label{lem a acts on vector 1} Suppose $\De = (\de, \ep)$ and $\De' = (\de', \ep') = (\mu'c^{r'}, \nu' d^{s'})$ are as in the statement of Theorem \ref{thm bandwidth},
    and that $\mu'$ properly extends $\ep$. Then for any $G \sse Z_{m'}$, \begin{align*}
        a_{\De} \cdot \eta_{\De',G} = \eta_{\De'', G}
    \end{align*}
    for some $\De'' \in S_{j', \ell'}$.
\end{Lemma}

\begin{proof}
For a set $G \sse Z_{m'}$,
\begin{align*}
    a_{\De} \cdot \eta_{\De', G} &= \chi_{[\de, \ep, Z_m]} \cdot \eta( \chi_{[\de', \ep', G]})\\
    &= \eta( \chi_{[\de, \ep, Z_m]} * \chi_{[\de', \ep', G]})\\
    &= \eta( \chi_{[\de, \ep, Z_m] \cdot [\de', \ep', G]})
\end{align*}
    Since $\mu'$ is a nontrivial extension of $\ep$, $\mu' = \ep \ze$ for some $|\ze| \ge 1$.
Then $\de' = \mu' c^{r'} = \ep \ze c^{r'}$, and
\begin{align*}
    [\de, \ep, Z_m] \cdot [\de', \ep', G] &= [\de, \ep, Z_m] \cdot [\ep \ze c^{r'}, \ep', G]\\
    &= [\de \ze c^{r'}, \ep \ze c^{r'}, Z_{m'}] \cdot [\ep \ze c^{r'}, \ep', G]\\
    &= [\de \ze c^{r'}, \ep', G].
\end{align*}
Let $\mu'' = \de \ze$. Then $(\mu'', \nu') \in S_{j'}$: first, $\mu''_{|\mu''|} = \ze_{|\ze|} = \mu'_{|\mu'|} \in \La_2^1$. Second, $|\mu''| = |\de| + |\ze| = |\de| + (|\mu'| - |\ep|) = |\mu'|$, because $|\de| = |\ep|$. And $(\mu', \nu') \in S_{j'}$.

Let $\de'' = \mu'' c^{r'}$. Then $\De'' := (\de'', \ep') \in S_{j', \ell'}$. Hence, $a_{\De} \cdot \eta_{\De', G} = \eta(\chi_{[\de'', \ep', G]}) = \eta_{\De'', G}$.
\end{proof}

This is our first result that describes what $a_{\De}$ does to one of the finite dimensional spaces $K_{\De'}(E)_{n+1}$.

\begin{Theorem}\label{thm a acts on K 1}
    Suppose $\De = (\de, \ep)$ and $\De' = (\de', \ep') = (\mu'c^{r'}, \nu' d^{s'})$ are as in the statement of Theorem \ref{thm bandwidth},
    and that $\mu'$ properly extends $\ep$. Let $n \geq m,$ and let $E \in Q_n^{m'}$.  Then
    \begin{align*}
        a_{\De} \cdot K_{\De'}(E)_{n+1} \sse K_{\De''}(E)_{n+1}
    \end{align*}
    for some $\De'' \in S_{j', \ell'}$.
\end{Theorem}

\begin{proof}
For any $F \in E_{n+1}^{m'}$,
\begin{align*}
    \langle \eta_{\De', F}, \eta_{\De', E} \rangle &= \tau(\chi_{[\de',\ep',E]}^* * \chi_{[\de', \ep', F]})\\
    &= \tau(\chi_{[\ep',\de',E]} * \chi_{[\de', \ep', F]})\\
    &= \tau(\chi_{[\ep',\de',E]\cdot[\de', \ep', F]})\\
    &= \tau(\chi_{[\ep',\ep',E\cap F]})\\
    &= \mf{m}(\ep'(E\cap F))\\
    &= \mf{m}(\ep' F). \stepcounter{equation}\tag{\theequation}\label{eq0 thm a acts on K 1}
\end{align*}
Let $u \in K_{\De'}(E)_{n+1}$. Write $u = \sum_{F \in E^{m'}_{n+1}} c_F \eta_{\De', F}$, where $c_F \in \C$.
By definition of $K_{\De'}(E)_{n+1}$, 
\begin{align*}
    0&= \langle u, \eta_{\De', E}\rangle\\
    &= \sum_{F \in E_{n+1}^{m'}} c_F \langle \eta_{\De', F}, \eta_{\De', E} \rangle\\
    &= \sum_{F \in E_{n+1}^{m'}} c_F \mf{m}(\ep' F) \stepcounter{equation}\tag{\theequation}\label{eq1 thm a acts on K 1}
\end{align*}
By Lemma \ref{lem a acts on vector 1}, since $E$ is a set contained in $Z_{m'}$, $a_{\De} \cdot \eta_{\De',E} = \eta_{\De'', E},$ and $\De'' \in S_{j',\ell'}$. Furthermore,
\begin{align*}
    a_{\De} \cdot u &= \sum_{F \in E^{m'}_{n+1}} c_F a_{\De} \cdot \eta_{\De', F}\\
    &= \sum_{F \in E^{m'}_{n+1}} c_F \eta_{\De'', F}.
\end{align*}
Now, 
\begin{align*}
    \langle a_{\De} \cdot u, \eta_{\De'',E} \rangle &= \sum_{F \in E^{m'}_{n+1}} c_F \langle \eta_{\De'', F}, \eta_{\De'',E}\rangle \\
    &= \sum_{F \in E^{m'}_{n+1}} c_F \tau(\chi_{[\ep', \de'', E]\cdot [\de'', \ep', F]})\\
    &= \sum_{F \in E^{m'}_{n+1}} c_F \tau(\chi_{[\ep', \ep', F]})\\
    &= \sum_{F \in E^{m'}_{n+1}} c_F \mf{m}(\ep' F).
\end{align*}
Thus by \eqref{eq1 thm a acts on K 1}, we have $\langle a_{\De}\cdot u, \eta_{\de'',E} \rangle = 0$. Hence, we have seen that $a_{\De}\cdot u \in \spn\{\eta_{\De'',F} : F \in E_{n+1}^{m'}\}$ and $a_{\De}\cdot u \perp \eta_{\De'', E}$. Therefore, $a_{\De}\cdot u \in K_{\De''}(E)_{n+1}$.
\end{proof}

For the next few results, we look at pairs $(\de,\ep) = (\mu \al^r, \nu \be^s)$ and $(\de',\ep') = (\mu' c^{r'}, \nu' d^{s'})$ where $\mu' = \nu$, and we consider the cases $c=\al$ and $c=\be$ in sequence. First we see what $a_{\De}$ does to a vector of the form $\eta_{\De', G}$ for $G \sse Z_{m'}$.

\begin{Lemma}\label{lem a acts on vector 2}
    Let $j \geq 1, \ell \geq 0, m = m(j,\ell), j' > j, \ell' \ge m,$ and $m' = m(j', \ell')$.
    Suppose $(\de, \ep) = (\mu\al^r, \nu\be^s) \in S_{j,\ell}$,
    and $(\de', \ep') = (\nu c^{r'}, \nu'd^{s'}) \in S_{j', \ell'}$, where $(\mu,\nu) \in S_j$ and $(\nu,\nu') \in S_{j'}$. Let $G \sse Z_{m'}$.

    Then there exist $p$ with $|p - \ell'| \leq m$, $\De'' = (\de'', \ep'') \in S_{j',p}$, and $G' \sse Z_{m(j',p)}$ such that $[\de, \ep, Z_m] \cdot [\de', \ep', G] = [\de'', \ep'', G']$.
    In particular, if $\De = (\de, \ep), \De' = (\de', \ep'),$ and $\De'' = (\de'', \ep'')$, then $a_{\De} \cdot \eta_{\De', G} = \eta_{\De'',G'}.$

    Specifically, 
    \begin{align*}
        p &= \begin{cases}
            \ell' - r &: c = \be\\
            \ell' + s &: c = \al,
        \end{cases}\\
        \De'' &= (\de'', \ep'') = \begin{cases}
            (\mu \be^{r'-s}, \nu' \al^{\ell'-r} ) &: c = \be\\
            (\mu \al^{r+r'}, \nu' \be^{\ell' + s} ) &: c = \al,
        \end{cases}\\
    \end{align*}
    and
    \begin{align*}
        G' &= \begin{cases}
            \al^r G &: c = \be\\
            \si^{\be^s}(G) &: c = \al.
        \end{cases}\\
    \end{align*}
\end{Lemma}

\begin{proof}
First we note that since $(\nu, \nu') \in S_{j',\ell'}$, and $|\nu| \le p_j < p_{j'}$, $|\nu'| = p_{j'}$ and so $s' = \ell'$. 

Suppose $c=\be$. Then
\begin{align*}
    [\de, \ep, Z_{m}]\cdot[\de', \ep', G] &= [\mu \al^r, \nu \be^s, Z_m] \cdot [\nu \be^{r'}, \nu' \al^{\ell'}, G]\\
    &= [\mu\al^r \be^{r'-s}, \nu\be^{r'}, Z_{m'}] \cdot [\nu \be^{r'}, \nu' \al^{\ell'}, G]\\
    &= [\mu\al^r \be^{r'-s}, \nu'\al^{\ell'}, G]\\
    &= [\mu\be^{r'-s}, \nu' \al^{\ell'-r}, \al^rG].
\end{align*}
Since $(\mu, \nu) \in S_j$ and $j \leq j'$, it follows that $|\mu| \leq p_j \leq p_{j'}$, thus $(\mu, \nu') \in S_{j'}$, and $|\nu' \al^{\ell' - r}| = p_{j'}+\ell' - r$.
We compute
\begin{align*}
    |\mu\be^{r'-s}| &= |\mu| + (p_{j'} + \ell' - |\nu|) - (p_j + \ell - |\nu|)\\
    &= (p_j + \ell - r) + p_{j'} + \ell' - p_j - \ell\\
    &= p_{j'} + \ell' - r.
\end{align*}
Hence, $(\mu\be^{r'-s}, \nu' \al^{\ell'-r}) \in S_{j', p}$ where $p = \ell'-r$. And $G' = \al^r G \sse Z_{m'-r}$. Since $m' - r = p_{j'} + \ell' + 1-r = m(j', \ell'-r)$, we have $G' \sse Z_{m(j',p)}$.

Next, suppose $c =\al$. Then
\begin{align*}
    [\de, \ep, Z_{m}]\cdot[\de', \ep', G] &= [\mu\al^r, \nu\be^s, Z_m] \cdot[\nu \al^{r'}, \nu'\be^{\ell'}, G]\\
    &= [\mu\al^{r+r'}, \nu\al^{r'} \be^s, Z_{m'+s}] \cdot [\nu\al^{r'}\be^s, \nu' \be^{\ell'+s}, \si^{\be^s}(G)]\\
    &= [\mu\al^{r+r'}, \nu'\be^{\ell'+s}, \si^{\be^s}(G)].
\end{align*}
As before, $(\mu, \nu') \in S_{j'}$, and $|\nu'\be^{\ell'+s}| = p_{j'} + \ell' + s$. We compute
\begin{align*}
    |\mu\al^{r+r'}| &= |\mu| + r + r'\\
    &= (p_j + \ell -r) + r + (p_{j'}+\ell' - |\nu|)\\
    &= p_j + \ell - |\nu| + p_{j'} +\ell'\\
    &= p_{j'} +  \ell' + s.
\end{align*}
So $(\mu\al^{r+r'}, \nu'\be^{\ell'+s}) \in S_{j', p}$ where $p= \ell'+s$. And $G' = \si^{\be^s}(G) \sse Z_{m'+s}$. Since $m'+s = p_{j'}+\ell' + 1 + s = m(j', \ell'+s)$, we have that $G' \sse Z_{m(j',p)}$.
\end{proof}

The next two lemmas and  subsequent corollary describe what follows if (borrowing the notation from Lemma \ref{lem a acts on vector 2}) $E'$ refines a given partition $Q_c^{(\cdot)}$.

\begin{Lemma}\label{lem if E' refines (c=beta)}
    Let $j, \ell, m, j', \ell', m',$ and $\De$ be as in Lemma \ref{lem a acts on vector 2}. Let $\De'=(\de', \ep') = (\nu \be^{r'}, \nu' \al^{\ell'}) \in S_{j', \ell'}$. Let $\de'' = \mu \be^{r'-s},$ $\ep''= \nu' \al^{\ell'-r},$ $p = \ell'-r$, and $\De''=(\de'', \ep'') \in S_{j', p}$, and let $m'' = m(j', p)$.

    Let $n \geq m$, fix $E \in Q_n^{m'}$, and let $u \in K_{\De}(E)_{n+1}$. If $\al^r E$ refines $Q_c^{m''}$, then for all $G \in Q^{m''}_c$, $a_{\De} \cdot u \perp \eta_{\De'', G}$.
\end{Lemma}

\begin{proof}
    For any $A \sse Z_{m'}$, let $A' = \al^r A$. 

    Say $u = \sum_{F \in E_{n+1}^{m'}} c_F \eta_{\De', F}$, where $c_F \in \C$ for each $F$. If $F \in E_{n+1}^{m'}$ then $F \sse E$, and so $F'\sse E'$. Note that Lemma \ref{lem a acts on vector 2} implies $a_{\De} \cdot u = \sum_{F \in E_{n+1}^{m'}} c_F \eta_{\De'', F'}$, and  $F' \neq \emp$ for all $F \in E_{n+1}^{m'}$, so the indices and coefficients are the same in both sums.
    
    Let $G \in Q_c^{m''}$. If $E' \cap G = \emp$ then for any $F \in E_{n+1}^{m'}$, $F' \cap G = \emp$ since $F' \sse E'$. So, $\eta_{\De'', F'} \perp \eta_{\De'', G}$, hence $a_{\De} \cdot u \perp \eta_{\De'', G}$.

    Otherwise, $E' \sse G$, so for any $F \in E_{n+1}^{m'}$, $F' \sse E' \sse G$. Hence, $F' \cap G = F' \cap E' = F'$, and so 
    \begin{align*}
        \langle \eta_{\De'', F'}, \eta_{\De'', G} \rangle &= \langle \eta_{\De'', F'}, \eta_{\De'', E'}\rangle\\
        &= \tau(\chi_{[\ep'', \de'', E']\cdot[\de'', \ep'', F']})\\
        &= \tau(\chi_{[\ep'', \ep'', F']})\\
        &= \mf{m}(\ep'' F').
    \end{align*}
    Then $\ep'' F' = \nu' \al^{\ell'-r} \al^r F = \nu' \al^{\ell'} F = \ep'F$. Thus for any $F \in E_{n+1}^{m'}$, \[\langle \eta_{\De'', F'}, \eta_{\De'', G}\rangle = \langle \eta_{\De'', F'}, \eta_{\De'', E'} \rangle = \mf{m}(\ep'F).\]
    Since $u \perp \eta_{\De', E}$ by definition of $K_{\De'}(E)_{n+1}$,
    \begin{align*}
        0 = \langle u, \eta_{\De',E}\rangle = \sum_{F \in E_{n+1}^{m'}} c_F \langle \eta_{\De', F}, \eta_{\De', E}\rangle = \sum_{F \in E_{n+1}^{m'}} c_F \tau(\chi_{[\ep', \ep', F]}) = \sum_{F \in E_{n+1}^{m'}} c_F \mf{m}(\ep'F).
    \end{align*}
    Hence,
    \begin{align*}
        \langle a_{\De} \cdot u, \eta_{\De'', G}\rangle = \sum_{F \in E_{n+1}^{m'}} c_F \langle\eta_{\De'', F'}, \eta_{\De'', G} \rangle = \sum_{F \in E_{n+1}^{m'}} c_F \mf{m}(\ep'F) = 0,
    \end{align*}
    and therefore $a_{\De} \cdot u \perp \eta_{\De'', G}$.
\end{proof}

\begin{Lemma}\label{lem if E' refines (c=alpha)}
    Let $j, \ell, m, j', \ell', m',$ and $\De$ be as in Lemma \ref{lem a acts on vector 2}. Let $\De'=(\de', \ep') = (\nu \al^{r'}, \nu' \be^{\ell'}) \in S_{j', \ell'}$. Let $\de'' = \mu \al^{r'+r}, \ep''= \nu' \be^{\ell'+s},$ $p = \ell'-r$, and $\De''=(\de'', \ep'') \in S_{j', p}$, and let $m'' = m(j',p)$.

    Let $n \geq m$, fix $E \in Q_n^{m'}$, and let $u \in K_{\De}(E)_{n+1}$. If $\si^{\be^s} (E)$ refines $Q_c^{m''}$, then for all $G \in Q^{m''}_c$, $a_{\De} \cdot u \perp \eta_{\De'', G}$.
\end{Lemma}

\begin{proof}
    For any $A \sse Z_{m'}$, let $A' = \si^{\be^s}(A)$. 

    Say $u = \sum_{F \in E_{n+1}^{m'}} c_F \eta_{\De', F}$, where $c_F \in \C$ for each $F$. If $F \in E_{n+1}^{m'}$ then $F \sse E$, and so $F'\sse E'$. Note that Lemma \ref{lem a acts on vector 2} implies $a_{\De} \cdot u = \sum_{F \in E_{n+1}^{m'}} c_F \eta_{\De'', F'}$, and  $F' \neq \emp$ for all $F \in E_{n+1}^{m'}$, so the indices and coefficients are the same in both sums.
    
    Let $G \in Q_c^{m''}$. If $E' \cap G = \emp$ then for any $F \in E_{n+1}^{m'}$, $F' \cap G=\emp$ since $F' \sse E'$. So, $\eta_{\De'', F'} \perp \eta_{\De'', G}$, hence $a_{\De} \cdot u \perp \eta_{\De'', G}$.

    Otherwise, $E' \sse G$, so for any $F \in E_{n+1}^{m'}$, $F' \sse E' \sse G$. Hence, $F' \cap G = F' \cap E' = F'$, and so 
    \begin{align*}
        \langle \eta_{\De'', F'}, \eta_{\De'', G} \rangle &= \langle \eta_{\De'', F'}, \eta_{\De'', E'}\rangle\\
        &= \tau(\chi_{[\ep'', \de'', E']\cdot[\de'', \ep'', F']})\\
        &= \tau(\chi_{[\ep'', \ep'', F']})\\
        &= \mf{m}(\ep'' F').
    \end{align*}
    Then $\ep'' F' = \nu' \be^{\ell'+s} \si^{\be^{s}}(F) = \nu' \be^{\ell'} F = \ep'F$. Thus for any $F \in E_{n+1}^{m'}$, \[\langle \eta_{\De'', F'}, \eta_{\De'', G}\rangle = \langle \eta_{\De'', F'}, \eta_{\De'', E'} \rangle = \mf{m}(\ep'F).\]
    Since $u \perp \eta_{\De', E}$ by definition of $K_{\De'}(E)_{n+1}$,
    \begin{align*}
        0 = \langle u, \eta_{\De',E}\rangle = \sum_{F \in E_{n+1}^{m'}} c_F \langle \eta_{\De', F}, \eta_{\De', E}\rangle = \sum_{F \in E_{n+1}^{m'}} c_F \tau(\chi_{[\ep', \ep', F]}) = \sum_{F \in E_{n+1}^{m'}} c_F \mf{m}(\ep'F).
    \end{align*}
    Hence,
    \begin{align*}
        \langle a_{\De} \cdot u, \eta_{\De'', G}\rangle = \sum_{F \in E_{n+1}^{m'}} c_F \langle\eta_{\De'', F'}, \eta_{\De'', G} \rangle = \sum_{F \in E_{n+1}^{m'}} c_F \mf{m}(\ep'F) = 0,
    \end{align*}
    and therefore $a_{\De} \cdot u \perp \eta_{\De'', G}$.
\end{proof}

\begin{Corollary}\label{coro if E' refines}
Let $j \geq 1, \ell \geq 0, m = m(j,\ell), j' \geq j, \ell' \ge m,$ and $m' = m(j', \ell')$.
    Suppose $(\de, \ep) = (\mu\al^r, \nu\be^s) $ with $|\de| = |\ep| = p_j + \ell,$ and $(\de', \ep') = (\nu c^{r'}, \nu'd^{s'}) \in S_{j', \ell'}$, where $(\mu,\nu) \in S_j$ and $(\nu,\nu') \in S_{j'}$. 

    Let $n \geq m$, fix $E \in Q_n^{m'}$, and let $u \in K_{\De}(E)_{n+1}$. Let $p$, $\De''$, and $E'$ be as in the conclusion of Proposition \ref{lem a acts on vector 2} for $G=E$, so that $a_{\De} \cdot \eta_{\De', E} = \eta_{\De'', E'}$, and $\De''\in S_{j',p}$. Let $m'' = m(j', p)$.
    
    If $E'$ refines $Q_k^{m''}$, then $$a_{\De} K_{\De'}(E)_{n+1} \perp \bigoplus_{t < k} \bigoplus_{A \in Q_t^{m''}} K_{\De''}(A)_{t+1}.$$
\end{Corollary}
\begin{proof}
    For any $A \in Q_{k-1}^{m''}$, $K_{\De''}(A)_k \sse \spn\{\eta_{\De'',B} : B \in Q_k^{m''}\}$, and by Lemmas \ref{lem if E' refines (c=beta)} and \ref{lem if E' refines (c=alpha)}, $a_{\De} u \perp \eta_{\De'',B}$ for any $B \in Q_k^{m''}$. Hence, $a_{\De} u \perp K_{\De''}(A)_k$ for any $A \in Q_{k-1}^{m''}$. So, $$a_{\De} u \perp \bigoplus_{A \in Q_{k-1}^{m''} } K_{\De''}(A)_k.$$
    
    In fact, since $E'$ refines $Q_k^{m''}$, it also refines $Q_t^{m''}$ for all $t \leq k$, so 
    $$a_{\De} u \perp \bigoplus_{t<k} \bigoplus_{A \in Q_{t}^{m''} } K_{\De''}(A)_{t+1}.$$
    
    Since $u \in K_{\De'}(E)_{n+1}$ was arbitrary, we have that
    $$a_{\De} K_{\De'}(E)_{n+1} \perp \bigoplus_{t<k} \bigoplus_{A \in Q_{t}^{m''} } K_{\De''}(A)_{t+1}.$$
\end{proof}

 The next results,  Propositions \ref{prop E' refines (c=beta)} and \ref{prop E' refines (c=alpha)}, say what partition $E'$ refines under different conditions. First we address the case wehre $c = \beta$, and then the case where $c = \al$.

\begin{Proposition}\label{prop E' refines (c=beta)}
    Let $j \geq 1, \ell \geq 0, m = m(j,\ell), j' \geq j, \ell' \ge m,$ and $m' = m(j', \ell')$. Let $\De = (\de, \ep) =(\mu\al^r, \nu\be^s) $ with $|\de| = |\ep| = p_j + \ell,$ and $\De' = (\de', \ep') = (\nu\be^{r'}, \nu' \al^{s'}) \in S_{j',\ell'}$, where $(\mu, \nu) \in S_j$ and $(\nu, \nu') \in S_{j'}$. Let $n \geq m$ and fix $E \in Q_n^{m'}$. Then $E' = \al^r E$ refines $Q_{n-m}^{m'-r}$.
\end{Proposition}

\begin{proof}
    If $E = \th R$ for some $\th \in \Phi_n^{m'}$ with $|\th| \geq 1$ and $R \in P_{m'+|\th|, n-|\th|}$, then $\al^r \th \in \Phi_{n+r}^{m'-r}$, and so $\al^rE \in Q_{n+r}^{m'-r}$. Since $Q_{n+r}^{m'-r}$ refines $Q_{n}^{m'-r}$, $\al^rE$ refines $Q_{n}^{m'-r}$.
    Otherwise, $E \in P_{m',n}$.

$P^{(1)}_{m',n}$:
If $E =Z_{m'}(\al^{n+1}\be^J)\sm Z_{m'}(\al^{n+1}\be^{J+1}) \in P^{(1)}_{m',n}$ for some $J \leq n$, then 
\begin{align*}
    E' = \al^rE = Z_{m'-r}(\al^{n+r+1}\be^J)\sm Z_{m'-r}(\al^{n+1+1}\be^{J+1})
\end{align*}
and $J \leq r+n$, so $E' \in P^{(1)}_{m'-r, n+r} \sse Q_{n+r}^{m'-r}$. Since $n+r \geq n,$ this refines $Q_{n}^{m'-r}$.

$P^{(2)}_{m',n}$: If $E =Z_{m'}(\al^{I}\be^{n+1})\sm Z_{m'}(\al^{I+1}\be^{n+1}) \in P^{(2)}_{m',n}$ for some $I \leq n$, then either $I+r \leq n$, or $I+r > n$.

If $I+r \leq n$, then $E'\in P^{(2)}_{m'-r, n}$, hence $E'$ refines $Q_n^{m'-r}$.

If $I+r > n$, then 
\begin{align*}
    E' &= Z_{m'-r}(\al^{I+r}\be^{n+1})\sm Z_{m'-r}(\al^{I+r+1}\be^{n+1})\\
    &\sse Z_{m'-r}(\al^{n+1}\be^{n+1})\\
    &\in P_{m'-r, n}^{(4)}.
\end{align*} And so $E'$ refines $Q_n^{m'-r}$.

$P^{(3)}_{m',n}$: If $E =Z_{m'}(\al^{I}\be^{J}\la) \in P^{(3)}_{m',n}$ for some $I, J \leq n \leq I + J$, and $\la \in \La_2^1$ with $r(\la) = v_{m' + I+J}$, then:

If $I+r \leq n$, then $E' \in P_{m'-r,n}^{(3)}$.

If $I+r > n$, then $E' \in P^{(3)}_{m'-r, I+r}$, so $E'$ refines $Q^{m'-r}_{I+r}$. Since $I+r > n$, this refines $Q^{m'-r}_{n}$.

$P^{(4)}_{m',n}$: If $E =Z_{m'}(\al^{n+1}\be^{n+1}) \in P^{(4)}_{m',n}$, then \begin{align*}
    E' &= Z_{m'-r}(\al^{r+n+1}\be^{n+1} \\
    &\sse Z_{m'-r}(\al^{n+1}\be^{n+1}) \\
    &\in P^{(4)}_{m'-r, n}.
\end{align*} So $E'$ refines $Q_n^{m'-r}$.

In all cases, $E'$ refines $Q_{n}^{m'-r}.$ Since $n -m \leq n$, this implies that $E'$ refines $Q_{n-m}^{m'-r}$.
\end{proof}

\begin{Proposition}\label{prop E' refines (c=alpha)}
    Let $j \geq 1, \ell \geq 0, m = m(j,\ell), j' \geq j, \ell' \ge m,$ and $m' = m(j', \ell')$. Let $\De = (\de, \ep) =(\mu\al^r, \nu\be^s) $ with $|\de| = |\ep| = p_j + \ell,$ and $\De' = (\de', \ep') = (\nu\al^{r'}, \nu' \be^{s'}) \in S_{j',\ell'}$, where $(\mu, \nu) \in S_j$ and $(\nu, \nu') \in S_{j'}$. Let $n \geq m$ and fix $E \in Q_n^{m'}$. Then $E' = \si^{\be^s}(E)$ refines $Q_{n-m}^{m'+s}$.
\end{Proposition}

\begin{proof}
    First we note that if $E \cap Z_{m'}(\be^s) = \emp$ then $E' = \emp$, which refines $Q_c^{m'+s}$ for all $c$. 

    If $E = \th R$ for some $\th \in \Phi_{n}^{m'}$ with $|\th| \geq 1$ and $R \in P_{m'+|\th|, n-|\th|}$, and $\th \La \cap \be^s \La = \emp$, then $E' = \emp$. 
    Otherwise, $\th \Cap \be^s$ and $E \sse Z_{m'}(\be^s)$. Since $\th_{|\th|} \in \La_2^1$, $\th$ must properly extend $\be^s$, so $\th = \be^s \th'$. Then $\th' \in \Phi_{n-s}^{m'+s}$, and $E' = \th' R \in Q_{n-s}^{m'+s}$. 

    Now suppose $E \in P_{m', n}$. Let $n' = n-s$ and $m''=m'+s$.

    $P^{(1)}_{m',n}$: If $E = Z_{m'}(\al^{n+1}\be^{J})\sm Z_{m'}(\al^{n+1}\be^{J+1}) \in P_{m',n}^{(1)}$ for $J < s$ then $E' = \emp$. So, suppose $E = Z_{m'}(\al^{n+1}\be^{J+s})\sm Z_{m'}(\al^{n+1}\be^{J+s+1}) \in P_{m',n}^{(1)}$ for some $J \leq n'$. Then $E' = Z_{m''}(\al^{n+1}\be^J)\sm Z_{m''}(\al^{n+1}\be^{J+1}) \in P_{m'',n}^{(1)}$. Hence $E'$ refines $Q_{n'}^{m''}$.

    $P^{(2)}_{m',n}$: If $E = Z_{m'}(\al^{I}\be^{n+1})\sm Z_{m'}(\al^{I+1}\be^{n+1}) \in P_{m',n}^{(2)}$ for some $I \leq n$, then $E'= Z_{m''}(\al^{I}\be^{n'+1})\sm Z_{m''}(\al^{I+1}\be^{n'+1})$.

    If $I \leq n'$, then $E' \in Q_{n'}^{m''}$, so $E'$ refines $Q_{n'}^{m''}$.

    If $I \geq n'+1$, then 
    \begin{align*}
        E' &\sse Z_{m''}(\al^{n'+1}\be^{n'+1})\\
        &\in P^{(4)}_{m'',n'}\\
        &\sse Q_{n'}^{m''}.
    \end{align*}
    So $E'$ refines $Q_{n'}^{m''}$.

    $P^{(3)}_{m',n}$: If If $E = Z_{m'}(\al^{I}\be^{J}\la) \in P_{m',n}^{(3)}$ where $J < s$, then $E' = \emp$. So we may assume that $E = Z_{m'}(\al^{I}\be^{J+s}\la) \in P_{m',n}^{(3)}$ for some $I, J$ such that $I, (J+s) \leq n \leq I+J+s$, and some $\la \in \La_2^1$ with $r(\la) = v_{m'+I+J+s}$. Then $E' = Z_{m''}(\al^I\be^J \la)$.

    If $I \leq n'$, then $I, J \leq n' \leq I+J$ and so $E' \in P_{m'',n'}^{(3)}$.

    If $I > n'$, then $J \leq I$ and so $E' \in P^{(3)}_{m'',I} \sse Q_I^{m''}$. Since $I > n'$, this refines $Q_{n'}^{m''}$.

    $P^{(4)}_{m',n}$: If $E = Z_{m'}(\al^{n+1}\be^{n+1})\in P_{m',n}^{(4)}$, then $E' = Z_{m''}(\al^{n+1} \be^{n'+1})$. So $E' \sse Z_{m''}(\al^{n'+1}\be^{n'+1}) \in P_{m'',n'}^{(4)}$. Hence $E'$ refines $Q_{n'}^{m''}$.

    So in all cases, $E'$ refines $Q_{n'}^{m''}$. Recall that $s = p_j + \ell - |\nu| \leq m$, so $n'= n-s \geq n-m$. Thus $Q_{n-s}^{m''}$ refines $Q_{n-m}^{m''}$.
\end{proof}

\begin{Theorem}\label{thm a acts on K 2}
Let $j \geq 1, \ell \geq 0, m = m(j,\ell), j' \geq j, \ell' \ge m,$ and $m' = m(j', \ell')$. Suppose $\De = (\de, \ep) = (\mu\al^r, \nu\be^s) $ with $|\de| = |\ep| = p_j + \ell,$ and $\De' =(\de', \ep') = (\nu c^{r'}, \nu'd^{s'}) \in S_{j', \ell'}$, where $(\mu,\nu) \in S_j$ and $(\nu,\nu') \in S_{j'}$. Let $n \geq 0$, and fix $E \in Q_n^{m'}$. Let

    \begin{tabular}{l l}
    $m'' = \begin{cases}
        m'-r &: c=\be\\m'+s &: c=\al,
    \end{cases}$
    & $\De'' = \begin{cases}
        (\mu \be^{r'-s}, \nu'\al^{\ell'-r}) &: c=\be\\ (\mu\al^{r+r'}, \nu'\be^{s+\ell'}) &: c=\al.
    \end{cases}$\\
\end{tabular}

Then \[ a_{\De} \cdot K_{\De'}(E)_{n+1} \perp \bigoplus_{t < n-m} \bigoplus_{A \in Q_{t}^{m''}} K_{\De''}(A)_{t+1}.\]
\end{Theorem}
\begin{proof}
    This follows from Lemmas \ref{lem if E' refines (c=beta)} and \ref{lem if E' refines (c=alpha)}, Corollary \ref{coro if E' refines}, and Propositions \ref{prop E' refines (c=beta)} and \ref{prop E' refines (c=alpha)}.
\end{proof}

Results \ref{lem if F is refined by 1} through \ref{coro if F is refined by 3} set us up for the refinement arguments found in Propositions \ref{prop F is refined by (c=beta)} and \ref{prop F is refined by (c=alpha)}. This set of results parallels the results from Lemma \ref{lem if E' refines (c=beta)} through Proposition \ref{prop E' refines (c=alpha)}.

\begin{Lemma}\label{lem if F is refined by 1}
    Let $j \geq 1, \ell \geq 0, m = m(j,\ell)$, and $F \sse Z_{m}$. Suppose $F$ is refined by $Q_k^m.$ Let $A \in Q_k^m$ with $A \sse F$, and let $w \in \spn\{\eta_{\De,B} : B \in A \cap Q_{k+1}^m \}$. Then $\langle w, \eta_{\De, F}\rangle = \langle w, \eta_{\De,A}\rangle$.
\end{Lemma}

\begin{proof}
    If $B \in A\cap Q_{k+1}^m$, then $B \sse A \sse F$, so $B \cap F = B\cap A$. Thus $\langle \eta_{\De,B}, \eta_{\De,F} \rangle = \langle \eta_{\De, B}, \eta_{\De,A}\rangle$, and so $\langle w, \eta_{\De, F}\rangle = \langle w, \eta_{\De,A}\rangle$.
\end{proof}

\begin{Corollary}\label{coro if F is refined by 2}
    Let $j \geq 1, \ell \geq 0, m = m(j,\ell)$, and $F \sse Z_{m}$. If $F$ is refined by $Q_k^m$, then for any $\De \in S_{j,\ell}$,
    \[ \eta_{\De,F} \perp \bigoplus_{t \geq k} \bigoplus_{A \in Q_t^m} K_{\De}(A)_{t+1}.\]
\end{Corollary}

\begin{proof}
    Let $A \in Q_k^m$, and let $w \in K_{\De}(A)_{k+1}$. If $A \sse F$ then $\langle \eta_{\De,F}, w\rangle = \langle \eta_{\De,A}, w\rangle$ by Lemma \ref{lem if F is refined by 1}, and $\langle \eta_{\De,A}, w\rangle = 0$ by definition of $K_{\De}(A)_{k+1}$.

    Otherwise, $F \cap A = \emp$, and so for all $B \in A \cap Q_{k+1}^m$, $F \cap B=\emp$. Thus $\langle \eta_{\De, F}, \eta_{\De, B} \rangle = 0$ and so $\langle\eta_{\De,F}, w\rangle = 0$.

    So we see that if $Q_k^m$ refines $F \sse Z_m$, then $\eta_{\De,F} \perp K_{\De}(A)_{k+1}$ for all $A \in Q_k^m$, hence $\eta_{\De,F} \perp \bigoplus_{A \in Q_k^m} K_{\De}(A)_{k+1}$. Furthermore, for all $t \ge k$, $Q_t^m$ refines $Q_k^m$ and so refines $F$, and it follows that 
    \[\eta_{\De,F} \perp \bigoplus_{t \geq k} \bigoplus_{A \in Q_t^m} K_{\De}(A)_{t+1}.\]
\end{proof}

\begin{Corollary}\label{coro if F is refined by 3}
    Let $j \geq 1, \ell \geq 0, m = m(j,\ell)$, and $F \sse Z_{m}$. If $\mc{F}$ is a finite collection of subsets of $Z_m$ and $\mc{F}$ is refined by $Q_k^m$, and $v \in \spn\{\eta_{\De,F}: F \in \mc{F}\}$ for some $\De \in S_{j,\ell}$, then
    \[ v \perp \bigoplus_{t \geq k} \bigoplus_{A \in Q_t^m} K_{\De}(A)_{t+1}.\]
\end{Corollary}
\begin{proof}
    This follows from Corollary \ref{coro if F is refined by 2}.
\end{proof}

The next two propositions state what partition refines the collection $\mc{F} = \{ F' : F \in E_{n+1}^{m'} \}$, where $F'$ is either $\al^r F$ in the case $c = \be$ (Prop. \ref{prop F is refined by (c=beta)}) or $\si^{\be^s} (F)$  in the case $c = \al$ (Prop. \ref{prop F is refined by (c=alpha)}).

\begin{Proposition}\label{prop F is refined by (c=beta)}
    Let $j \geq 1, \ell \geq 0, m = m(j,\ell), j' \geq j, \ell' \ge m,$ and $m' = m(j', \ell')$. Let $\De = (\de, \ep) =(\mu\al^r, \nu\be^s) $ with $|\de| = |\ep| = p_j + \ell$, and $\De' = (\de', \ep') = (\nu\be^{r'}, \nu' \al^{s'}) \in S_{j',\ell'}$, where $(\mu, \nu) \in S_j$ and $(\nu, \nu') \in S_{j'}$. Let $n \geq 0$ and fix $E \in Q_n^{m'}$.  Then the collection $\mc{F} = \{\al^r F : F \in E_{n+1}^{m'} \}$ is refined by $Q_{n+m}^{m'-r}$.
\end{Proposition}

\begin{proof}
    Let $m'' = m'-r$ and $n' = n+r$. For any $G \sse Z_{m'}$, let $G' = \al^r G$.
    
    If $E \in Q^{m'}_n \sm P_{m', n}$, so $E = \th R$ for some $\th \in \Phi_n^{m'}$ with $|\th| \ge 1$, and some $R \in P_{m'+|\th|, n-|\th|}$, then $E \cap Q_{n+1}^{m'} \sse Z(\th) \cap Q_{n+1}^{m'} = \th Q_{n+1 - |\th|}^{m'+|\th|}$. And since $\th \in \Phi_n^{m'}$, $\al^r \th \in \Phi_{n'}^{m''}$. Hence, $\mc{F} = \al^r (E^{m'}_{n+1}) \sse \al^r \th Q_{n+1 -|\th|}^{m'+|\th|} = Z(\al^r\th) \cap Q_{n'+1}^{m''}$. In other words, $\mc{F}$ is refined by $Q_{n'+1}^{m''}$.

    Otherwise, $E \in P_{m',n}$. This breaks down into several cases, each of which may also involve several cases...

    \textbf{Case 1:} $E = Z_{m'}(\al^{n+1}\be^J) \sm Z_{m'}(\al^{n+1}\be^{J+1}) \in P_{m',n}^{(1)}$ for some $J \leq n$. Then
    \begin{align*}
        E \cap (Q^{m'}_{n+1} \sm P_{m', n+1}) &= \emp\\
        E \cap P^{(1)}_{m',n+1} &= \{ Z_{m'}(\al^{n+2}\be^J) \sm Z_{m'}(\al^{n+2} \be^{J+1}) \}\\
        E \cap P^{(2)}_{m',n+1} &= \emp\\
        E \cap P^{(3)}_{m',n+1} &= \{ Z_{m'}(\al^{n+1}\be^J \la) : \la \in \La_2^1\} \\
        E \cap P^{(4)}_{m',n+1} &= \emp
    \end{align*}
    \begin{enumerate}
        \item If $F \in P^{(1)}_{m',n+1}$, then $F' = Z_{m''}(\al^{n'+2}\be^{J}) \sm Z_{m''}(\al^{n'+2}\be^{J+1}) \in P^{(1)}_{m'', n' + 1}$.

    \item If $F \in P^{(3)}_{m',n+1}$, then $F' = Z_{m''}(\al^{n'+1}\be^{J} \la) \in P^{(3)}_{m'', n' + 1}$.
    \end{enumerate}
    Thus $\mc{F}$ is refined by $Q^{m''}_{n'+1}$, which is refined by $Q^{m''}_{n+m}$.

    \textbf{Case 2:} $E = Z_{m'}(\al^{I}\be^{n+1}) \sm Z_{m'}(\al^{I+1}\be^{n+1}) \in P_{m',n}^{(2)}$ for some $I \leq n$. Then
    \begin{align*}
        E \cap (Q^{m'}_{n+1} \sm P_{m', n+1}) &= \emp\\
        E \cap P^{(1)}_{m',n+1} &= \emp \\
        E \cap P^{(2)}_{m',n+1} &= \{ Z_{m'}(\al^{I}\be^{n+2}) \sm Z_{m'}(\al^{I+1} \be^{n+2}) \}\\
        E \cap P^{(3)}_{m',n+1} &= \{ Z_{m'}(\al^{I}\be^{n+1} \la) : \la \in \La_2^1\} \\
        E \cap P^{(4)}_{m',n+1} &= \emp
    \end{align*}
    \begin{enumerate}
        \item If $F \in P^{(2)}_{m',n+1}$, then $F' = Z_{m''}(\al^{I+r}\be^{n+2})\sm Z_{m''}(\al^{I+r+1}\be^{n+2})$.
    \begin{enumerate}
        \item If $I+r \leq n+1$, then $F' \in P^{(2)}_{m'', n+1}$.
        \item Else, $n+1 < I+r$, and \begin{align*}
            F' = \left(\bigcup_{h=n+2}^{I+r} \bigcup_{\la \in \La_2^1} Z_{m''}(\al^{I+r}\be^h \la) \right) \cup Z_{m''}(\al^{I+r} \be^{I+r+1}) \sm Z_{m''}(\al^{I+r+1}\be^{I+r+1}).
        \end{align*}
        For $n+2 \leq h \leq I+r$ and $\la \in \La_2^1$, $Z_{m''}(\al^{I+r}\be^h \la) \in P^{(3)}_{m'', I+r}$. And $Z_{m''}(\al^{I+r} \be^{I+r+1}) \sm Z_{m''}(\al^{I+r+1}\be^{I+r+1}) \in P^{(2)}_{m'', I+r}$. Hence, $F'$ is refined by $Q_{I+r}^{m''}$. Since $I \leq n$ and $r \leq m$, $I+r \leq n+m$, so $F'$ is refined by $Q^{m''}_{n+m}$.
    \end{enumerate}

    \item If $F \in P^{(3)}_{m',n+1}$, then $F' = Z_{m''}(\al^{I+r}\be^{n+1}\la)$.
    \begin{enumerate}
        \item If $I+r \leq n+1$, $\al^r F \in P^{(3)}_{m'', n+1}$, which is refined by $Q_{n+m}^{m''}$.
        
        \item Otherwise, $F' \in P_{m'', I+r}^{(3)}$, and $I+r \leq n+m$, so $F'$ is refined by $Q_{n+m}^{m''}$.
    \end{enumerate}
    \end{enumerate}
    Thus $\mc{F}$ is refined by $Q^{m''}_{n+m}$.

    \textbf{Case 3:} If $E = Z_{m'}(\al^I \be^J \la) \in P^{(3)}_{m',n}$ for some $I, J \leq n \leq I+J$ and $\la \in \La_2^1$, let $\ze = \al^I \be^J \la$. Then
    \begin{align*}
        E \cap (Q^{m'}_{n+1} \sm P_{m', n+1}) &= \begin{cases}
            \emp &: I+J \geq n+1\\
            \ze P_{m'+n+1, 0} &: I+J=n
        \end{cases}\\
        E \cap P^{(1)}_{m',n+1} &= \emp\\
        E \cap P^{(2)}_{m',n+1} &= \emp\\
        E \cap P^{(4)}_{m',n+1} &= \emp \\
        E \cap P^{(3)}_{m',n+1} &= \begin{cases}
            \{E\} &: I+J \geq n+1\\ \emp &: I+J = n.
        \end{cases}\\
    \end{align*}
    \begin{enumerate}
        \item First suppose $I+J \ge n+1$. Then if $F \in E_{n+1}^{m'}$, $F = E = Z_{m'}(\al^I \be^J \la)$. So, $F' = Z_{m''}(\al^{I+r}\be^J \la)$. If $I+r \le n+1$, then $F' \in P^{(3)}_{m'', n+1}$, refined by $Q_{n+1}^{m''}$. If $n+1 < I+r$, then $F' \in P^{(3)}_{m'',I+r}$, and so is refined by $Q^{m''}_{I+r}$. Since $I+r \leq n+m$, $F'$ is refined by $Q_{n+m}^{m''}$.

        \item Suppose $I+J = n$. Then if $F \in E_{n+1}^{m'}$, $F \in \ze P_{m'+n+1, 0}$. Since $|\ze| = n + 1$, $\ze \in \Phi_{n+1}^{m'}$, and $\al^r \ze \in \Phi_{n'+1}^{m''}$. Hence, $\al^r \ze P_{m'+n+1,0} \sse Q_{n'+1}^{m''}$, and so $\mc{F}$ is refined by $Q_{n'+1}^{m''}$. Since $n'+1 \leq n+m$, $F'$ is refined by $Q_{n+m}^{m''}$.
    \end{enumerate}
    Therefore, $\mc{F}$ is refined by $Q_{n+m}^{m''}$.

    \textbf{Case 4:} $E = Z_{m'}(\al^{n+1}\be^{n+1}) \in P_{m',n}^{(4)}$. Then
    \begin{align*}
        E \cap (Q^{m'}_{n+1} \sm P_{m', n+1}) &= \emp\\
        E \cap P^{(1)}_{m',n+1} &= \{Z_{m'}(\al^{n+2}\be^{n+1}) \sm Z_{m'}(\al^{n+2}\be^{n+2})\} \\
        E \cap P^{(2)}_{m',n+1} &= \{ Z_{m'}(\al^{n+1}\be^{n+2}) \sm Z_{m'}(\al^{n+2} \be^{n+2}) \}\\
        E \cap P^{(3)}_{m',n+1} &= \{ Z_{m'}(\al^{n+1}\be^{n+1} \la) : \la \in \La_2^1\} \\
        E \cap P^{(4)}_{m',n+1} &= \{Z_{m'}(\al^{n+2}\be^{n+2})\}
    \end{align*}
    \begin{enumerate}
        \item If $F \in P^{(1)}_{m',n+1}$: \begin{align*}
            F' &= Z_{m''}(\al^{n'+2} \be^{n+1})\sm Z_{m''}(\al^{n'+2} \be^{n+2})\\
            &\in P^{(1)}_{m'', n'+1},
        \end{align*}
        so $F'$ is refined by $Q_{n'+1}^{m''}$.

        \item If $F \in P^{(2)}_{m',n+1}$:
        \begin{align*}
            F' &= Z_{m''}(\al^{n'+1}\be^{n+2}) \sm Z_{m''}(\al^{n'+2} \be^{n+2})\\
            &= \left(\bigcup_{h=n+2}^{n'+1} \bigcup_{\la \in \La_2^1} Z_{m''}(\al^{n'+1} \be^h \la) \right) \cup \left( Z_{m''}(\al^{n'+1} \be^{n'+2}) \sm Z_{m''}(\al^{n'+2} \be^{n'+2})\right).
        \end{align*}
        For $n+2 \le h \le n'+1$ and $\la \in \La_2^1$, $Z_{m''}(\al^{n'+1} \be^h \la) \in P^{(3)}_{m'', n'+1}$. And $Z_{m''}(\al^{n'+1} \be^{n'+2}) \sm Z_{m''}(\al^{n'+2} \be^{n'+2}) \in P^{(2)}_{m'', n'+1}$.

        Hence $F'$ is refined by $Q^{m''}_{n'+1}$.

        \item If $F \in P^{(3)}_{m',n+1}$: \begin{align*}
            F' &= Z_{m''}(\al^{n'+1}\be^{n+1}\la)\\
            &\in P^{(3)}_{m'', n'+1}\\
            &\sse Q_{n'+1}^{m''}.
        \end{align*}

        \item If $F \in P^{(4)}_{m',n+1}$: 
        \begin{align*}
            F' &= Z_{m''}(\al^{n'+2} \be^{n+2} ) \\
            &= \left( \bigcup_{h=n+2}^{n'+1} Z_{m''}(\al^{n'+2} \be^h) \sm Z_{m''}(\al^{n'+2}\be^{h+1}) \right) \cup Z_{m''}(\al^{n'+2} \be^{n'+2}).
        \end{align*}
        For $n+2 \le h \le n'+1$, $Z_{m''}(\al^{n'+2} \be^h) \sm Z_{m''}(\al^{n'+2}\be^{h+1}) \in P^{(1)}_{m'', n'+1}$. And $Z_{m''}(\al^{n'+2} \be^{n'+2}) \in P^{(4)}_{m'', n'+1}$.
        
        Hence $F'$ is refined by $Q_{n'+1}^{m''}$.
    \end{enumerate}
    Therefore, in all cases, $\mc{F}$ is refined by $Q_{n+m}^{m''}$.
\end{proof}

\begin{Proposition}
    \label{prop F is refined by (c=alpha)}
    Let $j \geq 1, \ell \geq 0, m = m(j,\ell), j' \geq j, \ell' \ge m,$ and $m' = m(j', \ell')$. Let $\De = (\de, \ep) =(\mu\al^r, \nu\be^s) $ with $|\de| = |\ep| = p_j + \ell,$ and $\De' = (\de', \ep') = (\nu\al^{r'}, \nu' \be^{s'}) \in S_{j',\ell'}$, where $(\mu, \nu) \in S_j$ and $(\nu, \nu') \in S_{j'}$. Let $n \geq m$ and fix $E \in Q_n^{m'}$.  Then the collection $\mc{F} = \{\si^{\be^s}(F) : F \in E_{n+1}^{m'} \}$ is refined by $Q_{n+m}^{m'+s}$.
\end{Proposition}
\begin{proof}
    Let $m'' = m'+s$ and $n' =n-s$. For any $G \sse Z_{m'}$, let $G' = \si^{\be^s}(G)$.

    If $\si^{\be^s}(E) = \emp$, then no extension of $\be^s$ is an element of $E$. Hence, the same must be true for any subset of $E$, so for any $F \in E_{n+1}^{m'}$, $F' = \emp$. So we shall assume $\si^{\be^s}(E) \neq \emp$.
    
    If $E \in Q^{m'}_n \sm P_{m', n}$, then $E = \th R$ for some $\th \in \Phi_n^{m'}$ with $|\th| \ge 1$, and some $R \in P_{m'+|\th|, n-|\th|}$. Since $\th_{|\th|} \in \La_2^1$ and $\th \Cap \be^s$, it follows that $\th \in \be^s \La$, so write $\th = \be^s \th'$. Thus $E \sse Z(\th)$, and so $E \cap Q_{n+1}^{m'} \sse Z(\th) \cap Q_{n+1}^{m'} = \th Q_{n+1 - |\th|}^{m'+|\th|}$. Hence if $F \in E^{m'}_{n+1}$, $F' \in \th' Q_{n+1-|\th'|}^{m'+|\th'|} \sse Q_{n+1-s}^{m'+s}$. And $n-s+1 \leq n+m$, so $\mc{F}$ is refined by $Q_{n+m}^{m''}$.

    Otherwise, $E \in P_{m',n}$. This breaks down into several cases, each of which may also involve several cases...

    \textbf{Case 1:} $E = Z_{m'}(\al^{n+1}\be^{J+s}) \sm Z_{m'}(\al^{n+1}\be^{J+s+1}) \in P_{m',n}^{(1)}$ for some $J \leq n'$. Then
    \begin{align*}
        E \cap (Q^{m'}_{n+1} \sm P_{m', n+1}) &= \emp\\
        E \cap P^{(1)}_{m',n+1} &= \{ Z_{m'}(\al^{n+2}\be^{J+s}) \sm Z_{m'}(\al^{n+2} \be^{J+s+1}) \}\\
        E \cap P^{(2)}_{m',n+1} &= \emp\\
        E \cap P^{(3)}_{m',n+1} &= \{Z_{m'}(\al^{n+1}\be^{J+s} \la) : \la \in \La_2^1\} \\
        E \cap P^{(4)}_{m',n+1} &= \emp
    \end{align*}
    \begin{enumerate}
        \item If $F \in P_{m', n+1}^{(1)}$, then $F' = Z_{m''}(\al^{n+2} \be^{J}) \sm Z_{m''}(\al^{n+2}\be^{J+1})$, which is in $P_{m'',n+1}^{(1)}$.

    \item If $F \in P^{(3)}_{m',n+1}$, then $F'= Z_{m''}(\al^{n+1}\be^{J}\la) \in P^{(3)}_{m'',n+1}$ for some $\la \in \La_2^1$.
    \end{enumerate}
    So $\mc{F}$ is refined by $Q_{n+1}^{m''}$, which is refined by $Q_{n+m}^{m''}$.

    \textbf{Case 2:} $E = Z_{m'}(\al^{I}\be^{n+1}) \sm Z_{m'}(\al^{I+1}\be^{n+1}) \in P_{m',n}^{(2)}$ for some $I \leq n$. Then
    \begin{align*}
        E \cap (Q^{m'}_{n+1} \sm P_{m', n+1}) &= \emp\\
        E \cap P^{(1)}_{m',n+1} &= \emp\\
        E \cap P^{(2)}_{m',n+1} &= \{Z_{m'}(\al^I\be^{n+2}) \sm Z_{m'}(\al^{I+1}\be^{n+2})\} \\
        E \cap P^{(3)}_{m',n+1} &= \{ Z_{m'}(\al^{I}\be^{n+1} \la) : \la \in \La_2^1\} \\
        E \cap P^{(4)}_{m',n+1} &= \emp
    \end{align*}

    If $F \in P^{(2)}_{m'',n+1}$, then $F' = Z_{m''}(\al^I \be^{n'+2}) \sm Z_{m''}(\al^{I+1}\be^{n'+2}).$ 
    \begin{enumerate}
        \item If $I \le n'+1$ then $F' \in P^{(2)}_{m'',n+1}$, and so is refined by $Q_{n+m}^{m''}$.

        \item Otherwise, $n'+1 < I \leq n$, and
        \begin{align*}
            F' &= Z_{m''}(\al^I \be^{n'+2}) \sm Z_{m''}(\al^{I+1}\be^{n'+2})\\
            &= \left( \bigcup_{h=n'+2}^{I-n'-1} \bigcup_{\la \in \La_2^1} Z_{m''}(\al^{I}\be^h \la) \right) \cup Z_{m''}(\al^I \be^{I+1}) \sm Z_{m''}(\al^{I+1} \be^{I+1}).
        \end{align*}
        For $n'+2 \leq h \leq I-n' -1$ and $\la \in \La_2^1$, $Z_{m''}(\al^I\be^h \la) \in P_{m'',I}^{(3)}$. And $Z_{m''}(\al^I \be^{I+1}) \sm Z_{m''}(\al^{I+1} \be^{I+1}) \in P_{m'', I}^{(2)}$. Hence $F'$ is refined by $Q_I^{m''}$, and since $I \le n$, this is refined by $Q_n^{m''}$.
    \end{enumerate}
    Therefore, $\mc{F}$ is refined by $Q_{n+m}^{m''}$.

    \textbf{Case 3:} $E = Z_{m'}(\al^{I}\be^{J+s} \la) \in P_{m',n}^{(3)}$ for some $I, J$ such that $I, (J+s) \leq n \leq I+J+s$ and some $\la \in \La_2^1$. Let $\ze=\al^I \be^{J+s} \la$. Then
    \begin{align*}
        E \cap (Q^{m'}_{n+1} \sm P_{m', n+1}) &= \begin{cases}
            \emp &: I+J+s \geq n+1\\
            \ze P_{m'+n+1,0}^{(2)} &: I+J+s = n
        \end{cases}\\
        E \cap P^{(1)}_{m',n+1} &= \emp\\
        E \cap P^{(2)}_{m',n+1} &= \emp\\
        E \cap P^{(3)}_{m',n+1} &= \begin{cases}
            \{E\} &: I+J+s \geq n+1\\ \emp &: I+J+s = n
        \end{cases}\\
        E \cap P^{(4)}_{m',n+1} &= \emp
    \end{align*}
    \begin{enumerate}
        \item First suppose $I+J+s \ge n+1$. Then if $F \in E_{n+1}^{m'}$, $F = E$, so $F' = E' \in P^{(3)}_{m'',n'}$, which is refined by $Q_{n+m}^{m''}$.

        \item Next suppose $I+J+s = n$. Then $I+J = n'$, so $I \leq n'$. If $F \in E_{n+1}^{m'}$, then $F \in \ze P_{m'+n+1, 0}$. Let $\ze' = \si^{\be^s}(\ze)$. Then $F' \in \ze' P_{m'+n+1,0}$, and \begin{align*}
        |\ze'| = |\ze|-s = (I+J+s + 1)-s = (n+1)-s = n'+1.
        \end{align*}
        Since $\ze'_{|\ze'|} = \la \in \La_2^1$, it follows that $\ze' \in \Phi_{n'+1}^{m''}$. And \[m' + n+1 = m' + s+n-s+1 = m'' + n' + 1 = m'' + |\ze'|,\] hence $P_{m'+n+1,0} = P_{m'' + |\ze'|, n'+1 - |\ze'|}$, so $F' \in \ze' P_{m'' + |\ze'|, n'+1 - |\ze'|} \sse Q_{n'+1}^{m''}$.
    \end{enumerate}
    Therefore $\mc{F}$ is refined by $Q_{n+m}^{m''}$.

    \textbf{Case 4:} $E = Z_{m'}(\al^{n+1}\be^{n+1}) \in P^{(4)}_{m', n}$. Then
    \begin{align*}
        E \cap (Q^{m'}_{n+1} \sm P_{m', n+1}) &= \emp\\
        E \cap P^{(1)}_{m',n+1} &= \{Z_{m'}(\al^{n+2}\be^{n+1}) \sm Z_{m'}(\al^{n+2}\be^{n+2})\} \\
        E \cap P^{(2)}_{m',n+1} &= \{ Z_{m'}(\al^{n+1}\be^{n+2}) \sm Z_{m'}(\al^{n+2} \be^{n+2}) \}\\
        E \cap P^{(3)}_{m',n+1} &= \{ Z_{m'}(\al^{n+1}\be^{n+1} \la) : \la \in \La_2^1\} \\
        E \cap P^{(4)}_{m',n+1} &= \{Z_{m'}(\al^{n+2}\be^{n+2})\}
    \end{align*}
    \begin{enumerate}
        \item If $F \in P^{(1)}_{m',n+1}$, then $F' = Z_{m''}(\al^{n+2}\be^{n'+1}) \sm Z_{m''}(\al^{n+2}\be^{n'+2}) \in P_{m'',n+1}^{(1)}$.

    \item If $F \in P^{(2)}_{m',n+1}$, then 
    \begin{align*}
        F' &= Z_{m''}(\al^{n+1}\be^{n'+2})\sm Z_{m''}(\al^{n+2}\be^{n'+1})\\
        &= \left( \bigcup_{h=n'+2}^{n+1} \bigcup_{\la\in\La_2^1} Z_{m''}(\al^{n+1}\be^{h}\la) \right) \cup \left( Z_{m''}(\al^{n+1}\be^{n+2})\sm Z_{m''}(\al^{n+2} \be^{n+2})\right).
    \end{align*}
    For $n'+2 \le h \le n+1$ and $\la \in \La_2^1$, $Z_{m''}(\al^{n+1}\be^{h}\la) \in P^{(3)}_{m'', n+1}$, and we have that $Z_{m''}(\al^{n+1}\be^{n+2})\sm Z_{m''}(\al^{n+2} \be^{n+2}) \in P^{(2)}_{m'', n+1}$.

    So $F'$ is refined by $Q_{n+1}^{m''}$.

    \item If $F \in P^{(3)}_{m',n+1}$, then $F' = Z_{m''}(\al^{n+1}\be^{n'+1} \la)$ for some $\la \in \La_2^1$, and so $F' \in P^{(3)}_{m'', n+1}$.

    \item If $F \in P^{(4)}_{m',n+1}$, then \begin{align*}
        F' &= Z_{m''}(\al^{n+2}\be^{n'+2})\\
        &= \left( \bigcup_{h=n'+2}^{n+1} Z_{m''}(\al^{n+2}\be^{h}) \sm Z_{m''}(\al^{n'+2}\be^{h+1}) \right) \cup Z_{m''}(\al^{n+2}\be^{n+2}).
    \end{align*}
    For $n'+2 \le h \le n+1$, $Z_{m''}(\al^{n+2}\be^{h}) \sm Z_{m''}(\al^{n'+2}\be^{h+1}) \in P^{(1)}_{m'',n+1}$. And $Z_{m''}(\al^{n+2}\be^{n+2}) \in P^{(4)}_{m'',n+1}$. So $F'$ is refined by $Q^{m''}_{n+1}$.
    \end{enumerate}
    Therefore, $\mc{F}$ is refined by $Q^{m''}_{n+m}$.
\end{proof}

\begin{Theorem}\label{thm a acts on K 3}
    Let $j \geq 1, \ell \geq 0, m = m(j,\ell), j' \geq j, \ell' \ge m,$ and $m' = m(j', \ell')$. Suppose $\De = (\de, \ep) = (\mu\al^r, \nu\be^s) $ with $|\de| = |\ep| = p_j + \ell,$ and $\De' =(\de', \ep') = (\nu c^{r'}, \nu'd^{s'}) \in S_{j', \ell'}$, where $(\mu,\nu) \in S_j$ and $(\nu,\nu') \in S_{j'}$. Let $n \geq m$, and fix $E \in Q_n^{m'}$. Let

    \begin{tabular}{l l}
    $m'' = \begin{cases}
        m(j', \ell'-r) &: c=\be\\
        m(j', \ell'+s) &: c=\al,
    \end{cases}$
    & $\De'' = \begin{cases}
        (\mu \be^{r'-s}, \nu'\al^{\ell'-r}) &: c=\be\\ (\mu\al^{r+r'}, \nu'\be^{s+\ell'}) &: c=\al.
    \end{cases}$\\
\end{tabular}

Then \[ a_{\De} \cdot K_{\De'}(E)_{n+1} \perp \bigoplus_{t \ge n+m} \bigoplus_{A \in Q_{t}^{m''}} K_{\De''}(A)_{t+1}.\]
\end{Theorem}

\begin{proof}
    If $F \sse Z_{m'}$, let $$F' = \begin{cases}
        \al^r F &: c = \be\\
        \si^{\be^s}(F) &: c = \al,
    \end{cases}$$
    and let $\mc{F} = \{ F' : F \in E_{n+1}^{m'}\}$.

    Suppose $u \in K_{\De'}(E)_{n+1}^{m'}$. Then $u \in \spn\{ \eta_{\De', F} : F \in E_{n+1}\}$. 
    By Lemma \ref{lem a acts on vector 2}, for all $F \in E_{n+1}^{m'}$, 
    $a_{\De} \cdot \eta_{\De', F} = \eta_{\De'', F'}$, hence $a_{\De} \cdot u \in \spn\{ a_{\De} \cdot \eta_{\De', F} : F \in E_{n+1}^{m'}\} = \spn\{ \eta_{\De', F'} : F' \in \mc{F}\}$. 

    By Propositions \ref{prop F is refined by (c=beta)} and \ref{prop F is refined by (c=alpha)}, the collection $\mc{F}$ is refined by $Q_{n+m}^{m''}$. So by Corollary \ref{coro if F is refined by 3}, 
    \begin{align*}
        a_{\De} \cdot u \perp \bigoplus_{t \ge n+m} \bigoplus_{A \in Q_t^{m''}} K_{\De''}(A)_{t+1}.
    \end{align*}
    Since $u \in K_{\De'}(E)_{n+1}$ was arbitrary, the result follows.
\end{proof}

Now the proof of Theorem \ref{thm bandwidth} is quite brief.
\begin{proof}\label{proof of thm bandwidth} \textit{(of \ref{thm bandwidth})}
    If $\mu'$ properly extends $\ep$, then Theorem \ref{thm a acts on K 1} says $a_{\De} \cdot K_{\De'}(E)_{n+1} \sse K_{\De''}(E)_{n+1}$, where $\De'' \in S_{j',\ell'}$.

    Now suppose $\mu' = \nu$. First, let
    \begin{align*}
        p &= \begin{cases}
            \ell'-r &: c=\be\\
            \ell'+s &: c=\al,
        \end{cases}\\
        \De'' &= \begin{cases}
            (\mu \be^{r'-s}, \nu'\al^{\ell'-r}) &: c=\be\\
            (\mu\al^{r+r'}, \nu'\be^{s+\ell'}) &: c=\al.
        \end{cases}
    \end{align*}
    Then $\De'' \in S_{j',p}$, and since $r, s \le m$, it follows that $\ell' - m \le p \le \ell'+m$.
    
    Let $\mc{F}$ be defined as in the proof of Theorem \ref{thm a acts on K 3}, so that $a_{\De} \cdot K_{\De'}(E)_{n+1} \sse \spn\{ \eta_{\De'', F'} : F' \in \mc{F}\}$. By Propositions \ref{prop F is refined by (c=beta)} and \ref{prop F is refined by (c=alpha)}, $\mc{F}$ is refined by $Q_{n+m}^{m(j',p)}$, which implies that for each $F' \in \mc{F}$, $\eta_{\De'',F'} \in \spn\{ \eta_{\De'',G} : G \in Q_{n+m}^{m(j',p)} \}$. Therefore,
    \begin{align*}
        a_{\De} K_{\De'}(E)_{n+1} \sse \spn\{ \eta_{\De'',G} : G \in Q_{n+m}^{m(j',p)} \}
    \end{align*}
    so (allowing for $c = \al$ or $\be$),
    \begin{multline*}
        a_{\De} K_{\De'}(E)_{n+1} \sse \spn\{ \eta_{\Ga,G} : \Ga \in S_{j',q}, \ G \in Q_{t}^{m(j',q)}, \ell-m \le q \le \ell+m,\\ \ n-m \le t \le n+m \}
    \end{multline*}
    
    By Theorem \ref{thm a acts on K 2},
    \begin{align*}
        a_{\De} \cdot K_{\De'}(E)_{n+1} \perp \bigoplus_{t < n-m} \bigoplus_{G \in Q_{t}^{m(j',p)}} K_{\De''}(G)_{t+1},
    \end{align*}
    and by Theorem \ref{thm a acts on K 3}
    \begin{align*}
        a_{\De} \cdot K_{\De'}(E)_{n+1} \perp \bigoplus_{t \ge n+m} \bigoplus_{G \in Q_{t}^{m(j',p)}} K_{\De''}(G)_{t+1}.
    \end{align*}
    Hence,
    \begin{align*}
        a_{\De} \cdot K_{\De'}(E)_{n+1} \sse \bigoplus_{t = n-m}^{n+m} \bigoplus_{G \in Q_{t}^{m(j',p)}} K_{\De''}(G)_{t+1}.
    \end{align*}
    Finally, allowing for $c= \al$ or $\be$, we obtain
    \begin{align*}
        a_{\De} \cdot K_{\De'}(E)_{n+1} \sse \bigoplus_{q=\ell'-m}^{\ell'+m}\ \bigoplus_{\Ga \in S_{j',q}}\ \bigoplus_{t = n-m}^{n+m}\ \bigoplus_{G \in Q_{t}^{m(j',q)}}\ K_{\Ga}(G)_{t+1}.
    \end{align*}
\end{proof}

\begin{Theorem}\label{thm total bandwidth}
    Recall $B_0= \{\chi_{[\mu c^r, \nu d^s, Z(s(\nu d^s))]} : (\mu, \nu) \in S_j$ for some $j \ge 0$, $|\mu c^r| = |\nu d^s|, \ \{c, d\} = \{\al,\be\} \}$ from Theorem \ref{thm total set for A}. For all $a \in B_0$, there exists $j, m \ge 1$ such that if $j' > j$ and $\ell, n \ge m$, then
    \begin{equation*}
        a \cdot H_{j', \ell, n} \sse \bigoplus_{q = \ell-m}^{\ell+m} \bigoplus_{t =n-m}^{n+m} H_{j', q, t}.
    \end{equation*}
\end{Theorem}

\begin{proof}
    If $a_{\De} \in B_0$ then $\De = (\mu \eta, \nu \th)$ where $(\mu, \nu) \in S_i$ for some $i \ge 0$ and $\eta, \th \in \La_1$. Let $j = \max\{i, 1\}$ and $m = |\mu \eta|$. Then let $j' > j$ and $\ell, n \ge m$.
    Definition \ref{def H_j l n} says
    \begin{align*}
        H_{j', \ell, n} &= \bigoplus_{\De \in S_{j',\ell}} \bigoplus_{E \in Q_n^{m(j',\ell)}} K_{\De}(E)_{n+1}.
    \end{align*}
    Thus Theorem \ref{thm bandwidth} implies
    \begin{align*}
        a \cdot H_{j',\ell,n} &\sse \bigoplus_{\De' \in S_{j',\ell}} \bigoplus_{E \in Q_n^{m(j',\ell)}} a_{\De} \cdot K_{\De'}(E)_{n+1}\\
        &\sse \bigoplus_{q=\ell-m}^{\ell+m} \bigoplus_{t=n-m}^{n+m} \bigoplus_{\De'' \in S_{j',q}} \bigoplus_{A \in Q_t^{m(j',t)}} K_{\De''}(A)_{t+1}\\
        &= \bigoplus_{q=\ell-m}^{\ell+m} \bigoplus_{t=n-m}^{n+m} H_{j', q, t}.
    \end{align*}
\end{proof}

\subsection{Defining the Unbounded Operator}

\begin{Lemma}\label{lem defining D}
    For $m = (m_1, m_2, m_3) \in \N^3$, let $c(m) = m_1 + 1 + m_2+ m_3$. Define an unbounded operator $D$ on $H$ by setting for $j \ge 1, \ell \ge 0,$ and $n \ge -2$, $D|_{H_{j,\ell,n}} = c(j-1, \ell, n+2) \mathds{1} = (j + \ell + n + 2)\mathds{1}$.
    Let $B_0$ be as in Theorem \ref{thm total set for A}. Then for all $a \in B_0$, $[D, a]$ extends to a bounded operator on $H$.
\end{Lemma}

\begin{proof}
We want to apply Proposition \ref{prop the point of D}, but with $\N$ replaced by $\N^3$, a directed set where $(m_1, m_2, m_3) \le (n_1, n_2, n_3)$ if and only if $m_i \le n_i$ for $i=1,2,3.$ We verify the necessary conditions on the scalar net:
\begin{itemize}
    \item For $i \in \N^3$, $1 \le c(i)$.
    \item If $i_1 =(j_1, \ell_1, n_1)$, $i_2 = (j_2, \ell_2, n_2) \in \N^3$ and $i_1 \le i_2$, $i_1 \neq i_2$, then $t_1 \le t_2$ for $t=j,\ell,n$, and at least one of $j_1, \ell_1, n_1$ is strictly less than its counterpart in $i_2$, so $c(i_1) <  c(i_2)$.
    \item If $k = (j_k, \ell_k, n_k) \in \N^3$, let $M_k = c(k)$. Let $i_0 = (j_0, \ell_0, n_0) \in \N^3$ be arbitrary. If $i = (j_i, \ell_i, n_i) \in \N^3$ is such that $i_0 - k \le i \le i_0 + k$, then 
    \begin{align*}
        |c(i) - c(i_0)| &= |j_i - j_0 +1 - 1 + \ell_i - \ell_0 + n_i - n_0|\\
        &\le |j_i - j_0| + 0 + |\ell_i - \ell_0| + |n_i - n_0|\\
        &\le j_k + \ell_k + n_k\\
        &< j_k + 1 + \ell_k + n_k\\
        &= M_k.
    \end{align*}
\end{itemize}
    If $a \in B_0$, then $a = \chi_{[\mu\eta, \nu\ze, Z_m]}$ where $(\mu,\nu) \in S_j$ for some $j \ge 0$ and $v_m = s(\nu\ze)$, and for all $n = (n_1, n_2, n_3) \in N^3$, if $n \ge (j+1, m, m),$ then, by Theorem \ref{thm total bandwidth},
    \begin{align*}
        a \cdot H_{n_1 + 1, n_2, n_3-2} \sse \bigoplus_{i = n-(0,m,m)}^{n+(0,m,m)} H_{i_1 + 1, i_2, i_3 - 2}.
    \end{align*}
    Thus, $[D,a]$ extends to a bounded operator on $H = \bigoplus_{n\in \N^3} H_{n_1+1, n_2, n_3-2}$.
\end{proof}

\begin{Theorem}\label{thm D is a Dirac operator}
    With $D$ as defined in Lemma \ref{lem defining D}, $(A, H, D)$ is a spectral triple.
\end{Theorem}

\begin{proof}
    First, if $a, b\in A$ and $[D, a]$ and $[D, b]$ extend to bounded operators on $H$, then $[D, ab] = [D,a]b + a[D, b]$. Since $a, b \in A \sse B(H)$, $[D,ab]$ also extends to a bounded operator on $H$. By induction, then, if $a_1, \dots, a_k \in A$ and each $[D, a_i]$ extends to a bounded operator on $H$, then so does $[D, a_1 \cdots a_k]$. And it is clear that if $\la \in \C$ then $[D, \la a + b] = \la [D, a] + [D,b]$, so $[D, \la a + b]$ extends to a bounded operator on $H$ as well.

    Let $B$ be the set consisting of all linear combinations of finite convolutions of elements of $B_0$. Then $B$ is dense $A$ by Theorem \ref{thm total set for A}, and the previous paragraph combined with Lemma \ref{lem defining D} imply that for all $a \in B$, $[D,a]$ extends to a bounded operator on $H$.

    Lastly, for $j \ge 1, \ell,n \ge 0,$ we see that $(1+D^2)^{-1}|_{H_{j,\ell,n}} = 1/(1 + c(j-1,\ell,n+2)^2)$. Let $T_{j,\ell,n} = \sum_{j' \le j, \ell' \le \ell, n' \le n} (1+D^2)^{-1}|_{H_{j', \ell', n'}}$, a finite rank operator, and $\|(1+D^2)^{-1} - T_{j,\ell,n}\| \to 0$ as $(j,\ell,n)$ increases: 
    \begin{multline*}
        \| (1+D^2)^{-1} - T_{j,\ell,n} \| =\\ \sup \{ 1/(1 + (j' + \ell' + n' + 2)^2) : \text{ at least one of } j' > j, \ell' > \ell, \text{ or } n' > n \}.
    \end{multline*}
    Hence by \cite[Theorem 4.4]{Con90}, $(1+D^2)^{-1}$ is compact.
\end{proof}

  \bibliographystyle{plain}

  \bibliography{references}

\end{document}